\documentclass[12pt,leqno]{article}
\usepackage{verbatim}
\usepackage{amssymb}
\usepackage{amsmath}
\usepackage{tabularx}
\usepackage{theorem}
\usepackage{xypic}

 \usepackage[mathscr]{eucal}

\newtheorem{theorem}[equation]{Theorem}
\newtheorem{corollary}[equation]{Corollary}

\newtheorem{lemma}[equation]{Lemma}

\newtheorem{proposition}[equation]{Proposition}

{\theorembodyfont{\rmfamily}

\newtheorem{definitionplain}[equation]{Definition}
\newtheorem{remark}[equation]{Remark}

\newtheorem{exampleplain}[equation]{Example}
}
\newcommand{\qed}{\hfill $\square$ \medskip}

\newenvironment{proof}[1][Proof]{\noindent\textbf{#1.} }{\
\rule{0.5em}{0.5em}\medskip}
\renewenvironment{proof}[1][Proof]{\noindent\textbf{#1.} }{\
\qed}

\newcommand{\Vogan}{}

\renewcommand{\exp}{\mathrm{exp}}

\newcommand{\Deltagh}{\Delta(\g,\h)}

\newcommand{\frh}{\mathfrak{h}}
\newcommand{\frg}{\mathfrak{g}}

\newcommand{\caH}{\mathcal{H}}

\newcommand{\caM}{\mathcal{M}}

\newcommand{\caB}{\mathcal{B}}

\newcommand{\caF}{\mathcal{F}}
\newcommand{\caS}{\mathcal{S}}

\newcommand{\bfB}{\mathbf{B}}
\newcommand{\bbZ}{\mathbb{Z}}

\newcommand{\Ztwo}{\Z/2\Z}
\newcommand{\wt}{\widetilde}
\newcommand{\lra}{\longrightarrow}
\newcommand{\bs}{\backslash}
\newcommand{\lam}{\lambda}
\newcommand{\frs}{{\mathfrak s}}

\newcommand{\R}{\mathbb R}

\newcommand{\C}{\mathbb C}
\newcommand{\Z}{\mathbb Z}

\newcommand{\G}{G_\C}

\newcommand{\tH}{\wt{H}}

\renewcommand{\b}{\mathcal B}
\newcommand{\tG}{\widetilde G}

\newcommand{\tK}{\widetilde K}

\newcommand{\tGprime}{\widetilde{G'}}

\newcommand{\ch}[1]{#1^\vee}

\newcommand{\ol}{\overline}

\renewcommand{\o}{\circ}

\newcommand{\SL}{\mathrm{SL}}
\newcommand{\SO}{\mathrm{SO}}
\newcommand{\Spin}{\mathrm{Spin}}
\newcommand{\SU}{\mathrm{SU}}
\newcommand{\U}{\mathrm{U}}
\newcommand{\GL}{\mathrm{GL}}

\newcommand{\Ext}{\mathrm{Ext}}
\newcommand{\irr}{\ol{\pi}}
\newcommand{\std}{\pi}

\newcommand{\h}{\mathfrak h}
\renewcommand{\k}{\mathfrak k}

\newcommand\inv{^{-1}}

\newcommand{\g}{\mathfrak g}

\newcommand{\fra}{{\mathfrak a}}
\newcommand{\fru}{{\mathfrak u}}
\newcommand{\frw}{{\mathfrak w}}
\newcommand{\frn}{{\mathfrak n}}
\newcommand{\frB}{{\mathfrak B}}

\newcommand{\Ind}{\mathrm{Ind}}

\newcommand{\Norm}{\text{Norm}}

\newcommand{\Lift}{\text{Lift}}

\newcommand\brach[2]{\langle#1,\ch#2\rangle}

\newcommand{\ghalf}{g_{\frac12}}

\renewcommand{\sec}[1]{\section{#1}
\renewcommand{\theequation}{\thesection.\arabic{equation}}
  \setcounter{equation}{0}}
\newcommand{\subsec}[1]{\subsection{#1}
\renewcommand{\theequation}{\thesection.\arabic{equation}}}

\newcommand{\Lie}{\mathrm{Lie}}
\newcommand{\supp}{\mathit{supp}}

\newcommand{\scB}{\mathscr{B}}

\begin{document}
\title{Duality for Nonlinear Simply Laced Groups}
\author{
Jeffrey Adams\footnote{~{\tt jda@math.umd.edu}, Department of Mathematics, University of Maryland, College Park, MD 20742.  Partially supported 
by NSF DMS-0532393 and DMS-0554278.}
~and Peter E.~Trapa\footnote{~{\tt ptrapa@math.utah.edu}, Department of Mathematics, University of Utah, Salt Lake City, UT 84112. Partially supported 
by NSF DMS-0554118 and DMS-0554278.}
}
\date{}
\maketitle

\sec{Introduction}
\label{s:introduction}

The goal of this paper is to extend some of the
formalism of the Local Langlands Conjecture to certain nonlinear (that
is, nonalgebraic) double covers of real groups.
Such nonlinear groups, for example the two-fold cover of the symplectic
group, have long been known to play an interesting role in the theory of
automorphic forms of algebraic groups.

The Langlands formalism applies to the real points $G$ of a connected,
reductive complex group $G_\C$, and the dual group $\ch G_\C$ plays a
central role.
An important first step in extending Langlands' formalism to nonlinear
groups is to find the correct analog of the dual group $\ch G_\C$. There
is no obvious natural candidate for it.  We consult the refinement of
the local Langlands conjecture given in \cite{vogan:llc}.  The
viewpoint adopted there, originating in work of Deligne, Kazhdan,
Langlands, Lusztig and others, is that the right dual object is a
complex algebraic variety $X$ equipped with an action of the complex
group $\ch G_\C$. The $\ch G_\C$ orbits on $X$ should capture
detailed information about the representation theory of $G$.  

The space $X$ was first defined by Langlands \cite{langlands} 
as the space of admissible homomorphisms of the Weil group into the
L-group of $G$.
In the case of the real field the space was  revised in \cite{abv}.  With the
revision in place, something remarkable happens: $\ch G_\C$-orbits on
$X$ also parametrize (packets of) representations of real forms of the dual
group $\ch G_\C$.  The representation theoretic part of the
Local Langlands Conjecture may then be interpreted as
as a duality, sometimes referred to as Vogan duality,
between representations of real forms of
$G_\C$ and those of $\ch G_\C$.  See \cite[Theorem 10.4]{abv}.
In this form the duality reduces in large
part to \cite[Theorem 1.15]{ic4}.

This suggests the possibility of
extending the duality theory of \cite{ic4} to nonlinear groups as a
first step in  extending the Langlands formalism to them.  Some of
the foundations of this approach were given in 
\cite{rt1}--\cite{rt3}.  In particular, these papers treated
all simple groups of type $A$ as well as the metaplectic double
cover of the symplectic group (although many of the arguments were
case-by-case).  The purpose of the present paper is to give a unified
duality theory for all nonlinear double covers of simply laced real
reductive groups.

To make these ideas more precise we briefly describe
the main result of \cite{ic4}. Let $\b$ be a block of representations
of $G$ with regular infinitesimal character; this is a finite set of
irreducible representations.  (For terminology related to blocks and
infinitesimal
characters, see the discussions at the beginning of Section \ref{s:dual}
and preceding Theorem \ref{t:lang} respectively.)
Write $\caM$ for the $\Z$-module
spanned by the elements of $\b$ viewed as a submodule of the
Grothendieck group.  Then $\caM$ has two
distinguished bases: the irreducible representations in $\b$, and a
corresponding basis of standard modules.  We may thus consider the
matrix relating these two bases.  (The Kazhdan-Lusztig-Vogan
algorithm of \cite{ic3} provides a means to compute this matrix.)  We
say a block $\b'$ of representations is {\em dual} to $\b$ if there is
a bijection between $\b$ and $\b'$ for which, roughly speaking, the
two corresponding change of basis matrices for $\b$ and $\b'$ are
inverse-transposes of each other.  See Remark \ref{r:hecke} for a stronger
formulation in terms of certain Hecke modules.

Suppose the infinitesimal character $\lambda$ of $\b$ is 
integral. According to \cite[Theorem 1.15]{ic4} 
there is a real form $\ch G$ of $\ch\G$ and a block $\ch\b$ for $\ch
G$, so that $\ch\b$ is dual to $\b$. 
If $\lambda$ is not integral the
same result holds with $\ch\G$ replaced by a subgroup of $\ch\G$.
See Section \ref{s:defofdual} for more detail.
(The case of singular infinitesimal character introduces some
subtleties; see Remark \ref{rem:singular}.)

All of the ingredients entering the statement of \cite[Theorem
1.15]{ic4} still make sense for nonlinear groups.  (The computability
of the change basis matrix and Hecke module formalism in the nonlinear setting is
established in \cite{rt3}.)  Thus we seek to establish a similar
statement in the nonlinear case.  Our main result is:

\begin{theorem}[cf. Theorem \ref{t:main}] \label{t:mainintro} Suppose $G$
is a real reductive linear group; see Section \ref{s:notation} for
precise assumptions.  We assume the root system of $G$ is simply
laced.  Suppose $\tG$ is an admissible two-fold cover of $G$
(Definition \ref{d:admissible}).  Let $\b$ be a block of genuine
representations of $\tG$, with regular infinitesimal character. Then
there is a real reductive linear group $G'$, an admissible two-fold
cover $\tGprime$ of $G'$, and a block of genuine representations of
$\tGprime$ such that $\b'$ is dual to $\b$.  
\end{theorem}

Often, but not always, $\tGprime$ is a nonlinear group. 
Let $\Delta$ be the root system of $G_\C$.
The root system $\Delta'$ of $G'_\C$ is a subsystem of the dual root system
$\ch\Delta$, with equality if the infinitesimal character of $\b$ is
{half-integral} (Definition \ref{d:halfint}).
However this is misleading: one should think of $\Delta'$ as a root
subsystem of $\Delta$. This is possible since $\Delta$ is simply laced,
so $\Delta\simeq\ch\Delta$, and this isn't a meaningful distinction.
However this distinction {does} of course appear in the non simply-laced case.
For example if $\tG$ is the two-fold cover
$\wt{Sp}(2n,\R)$ of $Sp(2n,\R)$, then  (in the half-integral case)
$\tG'$ is also
$\wt{Sp}(2n,\R)$, rather than a cover of $SO(2n+1)$
\cite{rt1}.

Theorem \ref{t:main}
together with \cite{rt2} give a duality theory for double covers of all
simple, simply connected groups except those of type $B_n$, $F_4$,
and $G_2$. 
(The case of $G_2$ is partially covered by this paper, and is not
difficult to work out by hand; see 
Example \ref{e:g2}. In addition S.~Crofts has made substantial progress on groups
of type $B$ and $F_4$ \cite{crofts}.)
However, as is also true in \cite{ic4}, treatment of
simple groups is inadequate to handle general reductive groups.  In
particular Theorem \ref{t:main} cannot be reduced to the case of
simple groups.

It is worth noting that the genuine representation theory
of a simply laced nonlinear group $\tG$ is in many ways much simpler
than that of a linear group.  For example $\tG$ has at most one genuine
discrete  series representation with given infinitesimal and central
characters. 
The same result holds for principal series of a split
group; this is the dual of the preceding assertion. See Examples
\ref{ex:verysimple} and \ref{ex:verysimple2}.  
Furthermore, while disconnectedness is a major
complication in the proofs in \cite{ic4}, it plays essentially
no role in this paper.

Finally, there is a close relationship between the duality of Theorem
\ref{t:mainintro} and the lifting of characters developed in
\cite{aherb}.  This is explained at the beginning of Section
\ref{s:defofdual}.

\bigskip

\noindent{\bf Acknowledgments.}
The authors would like to thank David Renard and David Vogan for a
number of helpful discussions during the course of this project, and 
Rebecca Herb for help with a number of technical issues.

\sec{Some notation and structure theory}
\label{s:notation}

\bigskip

A \textit{real form} $G$ of a complex algebraic group $G_\C$ is
the fixed points of an antiholomorphic involution $\sigma$ of $G_\C$.  

By a {\it real reductive linear group} we mean a group in the category
defined in \cite[Section 0.1]{green}. Thus $\g=\Lie_\C(G)$ is a
reductive Lie algebra, $G$ is a real group in Harish-Chandra's class
\cite[Definition 4.29]{knapp_vogan},
has a faithful finite-dimensional
representation,
and has abelian Cartan subgroups.
Examples include real forms of
connected complex reductive groups, and any subgroup of finite index
of such a group.

We will sometimes work in the greater generality of a
{\it real reductive group in Harish-Chandra's class}:
$\g$ is reductive and $G$ is a real group in 
Harish-Chandra's class.
Examples include any two-fold cover of a
real reductive linear group (see Section \ref{s:nonlinear}).
A Cartan subgroup in $G$ is defined to be the
centralizer of a Cartan subalgebra, and may be 
nonabelian.

We will always assume we have chosen a Cartan involution $\theta$ of
$G$ and let $K=G^\theta$, a maximal compact subgroup of $G$.  Fix a
$\theta$-stable Cartan subgroup $H$ of $G$.  Write $H=TA$ with
$T=H\cap K$ and $A$ connected and simply connected.  We denote the Lie
algebras of $G,H,K\dots$ by $\g_0,\h_0,\k_0$, and their
complexifications by $\g,\h,\k$.  The Cartan involution acts on the
roots $\Delta=\Deltagh$ of $\h$ in $\g$. We use $r,i,cx, c,n$ to
denote the real, imaginary, complex, compact, and noncompact roots,
respectively (as in \cite{green}, for example).  We say $G$ is simply
laced if all roots in each simple factor of $\Delta$ have the same
length, and we declare all the roots to be long in this case.  We only
assume $G$ is simply laced when necessary.  We denote the center of 
$G$ by $Z(G)$, and the identity component by $G^0$.

\sec{Nonlinear Groups}
\label{s:nonlinear}

We introduce some notation
and basic results for two-fold covers.
Throughout this section $G$ is a real reductive linear group 
(Section \ref{s:notation}).

\begin{definitionplain}
\label{d:cover}
We say a real Lie group
$\tG$ is a two-fold cover of $G$ if $\tG$ is a central extension of
$G$ by $\Z/2\Z$.  Thus there is an exact sequence of Lie groups
$$
1\rightarrow A\rightarrow \tG\rightarrow G\rightarrow 1
$$
with $A\simeq\Z/2\Z$ central in $\tG$.
We say $\tG$ is nonlinear if it is not a
linear group, i.e.~$\tG$ does not admit a faithful finite-dimensional
representation.
A {\it genuine} representation of $\tG$ is one which does not factor
to $G$.  If $H$ is a subgroup of $G$, we let $\tH$
denote the inverse image of $H$ in $G$.
\end{definitionplain}

\medskip

Suppose $H$ is a $\theta$-stable Cartan subgroup of $G$, 
and $\alpha$ is a real or imaginary root. Corresponding to
$\alpha$ is the subalgebra $\mathfrak m_\alpha$ of $\g$ generated by the root vectors
$X_{\pm\alpha}$. The corresponding
subalgebra of $\g_0$ is isomorphic to
$\mathfrak s\mathfrak u(2)$ if $\alpha$ is imaginary and compact, or 
$\mathfrak
s\mathfrak l(2,\R)$ otherwise.
Let $M_\alpha$ be the corresponding analytic subgroup of $G$.
The following definition is fundamental for the study of
nonlinear groups and appears in many places.

\begin{definitionplain}
\label{d:meta}
Suppose $\tG$ is a two-fold cover of $G$.
Fix a $\theta$-stable Cartan subgroup $H$ of $G$
and a real or noncompact imaginary root $\alpha \in\Deltagh$.  
We say that $\alpha$ is {\em metaplectic} if $\wt{M_\alpha}$
is a nonlinear group, i.e.~is the nontrivial two-fold cover of
$\SL(2,\R)$.
\end{definitionplain}

See \cite[Section 1]{groupswithcovers} for the next result. Part (1)
goes back to \cite{steinberg_cocycle}; also see
\cite[Section 2.3]{mirkovic}.

\begin{lemma}
\label{l:metaplecticRoots}
\begin{enumerate}
\item 
Fix a two-fold cover $\tG$ of a real reductive linear  group $G$ and
assume $\g_0$ is simple.  Suppose $H$ is a Cartan subgroup of $G$ and
$\alpha\in\Deltagh$ is a long real or long noncompact imaginary root.
Then $\tG$ is nonlinear if and only if $\alpha$ is metaplectic.

\item 
Suppose further that $G$ is a real form of a connected, simply
connected, simple complex group. Then $G$ admits
a nonlinear two-fold cover if and only if
there is a $\theta$-stable Cartan subgroup $H$ with a long real root. 
The same conclusion holds with noncompact
imaginary  in place of real. This cover is unique up to
isomorphism.

\item Suppose  $G$ is as in (2), and also assume it is simply laced.
Let $H$ be any $\theta$-stable Cartan subgroup of $G$.
Then $G$ admits a nonlinear two-fold cover if and only if there is a real or noncompact
imaginary root $\alpha\in\Deltagh$.
\end{enumerate} 
\end{lemma}

\begin{definitionplain}
\label{d:admissible}
We say a two-fold cover $\tG$ of a real reductive linear group $G$ 
is {\it admissible} if for every
$\theta$-stable Cartan subgroup $H$  every  long real or long
noncompact imaginary root is metaplectic.
\end{definitionplain}

\begin{lemma}
\label{l:admissible}
Let $G$ be a real  reductive linear  group, and 
retain the notation of Definition \ref{d:cover}.  Let $G_1, \dots, G_k$
denote the analytic subgroups of $G$ corresponding to the simple
factors of the Lie algebra of $G$.
A two-fold cover $\tG$ of $G$ is admissible if and only if  
$\wt G_i$ is nonlinear for each $i$ such that $G_i$ admits
a nonlinear cover.
\end{lemma}

This is immediate from  Lemma \ref{l:metaplecticRoots}(1) and
the definitions.  Roughly speaking, the result says admissible
covers are as nonlinear as possible.

\begin{proposition}
\label{p:admissiblecovers}
Assume $G$ is a real form of a connected, simply connected,
semisimple complex algebraic group $\G$.
Then $G$ admits an
admissible two-fold cover, which is  unique up to isomorphism.
\end{proposition}

\begin{proof}
It is well-known (for example see \cite[Theorem
2.6(2)]{platonov_rapinchuk})
that 
$\G\simeq\prod_{i=1}^nG_{i,\C}$ (direct product)
with each $G_{i,\C}$ simple, and
$G\simeq \prod_{i=1}^n G_i$ accordingly.
Let $\overline G=\prod \tG_i$ where $\tG_i$ is the unique nontrivial
two-fold cover of $G_i$, if it exists, or the trivial cover
otherwise. (cf.~Lemma \ref{l:metaplecticRoots}(2).)
Then $\overline G$ has a natural quotient which is an
admissible cover of $G$.

Conversely if $\tG$ is an admissible cover of $G$ then
$G=\prod\wt{G_i}$ (the product taken  in $\tG$; the terms $\tG_i$
necessarily commute). Then $\tG$ is a quotient of the group $\overline
G$ constructed above, is isomorphic to the cover constructed in the
previous paragraph. 
\end{proof}

% If $G$ is reductive there may be an  obstruction to the existence of an
% admissible cover, although this is rare. The following result along
% these lines will be sufficient for our purposes.

% \begin{proposition}
% Suppose $\g_0$ is a real reductive Lie algebra. There is a real  reductive linear
% group $G$ with $\Lie(G)=\g_0$ which admits an admissible two-fold cover.
% \label{p:coverexists}
% \end{proposition}

% \begin{proof}
% Let $G(\C)=T(\C)\times G_d(\C)$, where $T(\C)$ is a torus, $G_d(\C)$
% is simply connected and semisimple, and
% $\Lie(G(\C))=\g_0\otimes\C$. Let $G=T\times G_d$ be the real form of
% $G(\C)$ corresponding to $\g_0$.  By Proposition
% \ref{p:admissiblecovers} $G_d$ has an admissible cover $\wt{G_d}$:
% take $\tG=T\times \wt{G_d}$.
% \end{proof}

% \begin{comment}
% \begin{remark}
% Assume the derived group $G_d$ is connected and has an admissible two-fold
% cover; for example this holds if $G$ is the real points of a simply
% connected semisimple group.
% It is not hard to show that $G$ has an admissible two fold cover 
% if $G=Z(G)G_d$ (for example $G=U(1,1)$ or if $G$ is split
% (e.g. $G=GL(2,\R)$).
% \end{remark}
% \end{comment}

\begin{exampleplain}
\hfil
\begin{enumerate}
\item  Every  two-fold cover of an abelian Lie group is admissible.
\item The group $GL(n,\R)$ ($n\ge 2$) has two admissible covers, up
  to isomorphism, each
  with the same nontrivial restriction to $SL(n,\R)$.
These correspond to the two two-fold covers of $O(n)$, sometimes
denoted $Pin^\pm$.
The $\sqrt{\det}$ cover of $GL(n)$ is not admissible.

\item The group $U(p,q)$ $(pq>0)$ has three inequivalent admissible
covers, corresponding to the three nontrivial
two-fold covers of $K=U(p)\times U(q)$, all having the same
nontrivial restriction to $SU(p,q)$.

\item Suppose each analytic subgroup of $G$ corresponding to a simple
factor of $\g_0$ admits no nonlinear cover.  (For example, suppose
each such analytic subgroup is compact, complex, or $\Spin(n,1)$ for
$n\geq 4$.)  Then every two-fold cover of $G$ is admissible, and nonlinear.

\end{enumerate}
\end{exampleplain}

\sec{Definition of Duality}
\label{s:dual}

We recall some definitions and notation from \cite{green} and
\cite{ic4}.
In this section we let $G$ be any real reductive group in Harish-Chandra's
class (cf.~Section \ref{s:notation}).

Recall that block equivalence is the relation on irreducible modules
generated by $\pi \sim \eta$ if $\Ext^1(\pi,\eta)$ is nonzero.  See
\cite[Section 9.1]{green}.  
It is easy to see that all elements in
a block have the same infinitesimal character, which we refer to as
the infinitesimal character of the block, and therefore each block is
a finite set. If $\pi$ is an irreducible representation with regular
infinitesimal character we write $\b(\pi)$ for the block containing
$\pi$. 

Fix a block $\b$ with regular infinitesimal character.  
Theorem \ref{t:lang} below provides a parameter set $\ol\scB$
and for each $\gamma\in\ol\scB$ 
a standard representation $\pi(\gamma)$, with a unique irreducible
quotient $\ol\pi(\gamma)$. 
The map from $\ol\scB$ to $\b$ given by $\gamma\rightarrow
\overline\pi(\gamma)$ 
is a bijection.

Given a Harish-Chandra module $X$, we let $[X]$ denote its image
in the 
Grothendieck group of all Harish-Chandra modules. 
Consider the $\Z$-module $\Z[\b]$ with basis $\{[\pi]\in\b\}=\{ [\irr(\gamma)] \,
|\,\gamma\in\ol\scB\}$, viewed as a subgroup of the
Grothendieck group.
Then $\Z[\b]$ is also spanned by the standard modules
$\{[\std(\gamma)]\,|\,\gamma \in\ol\scB\}$.  We may thus consider the
change of basis matrices in the Grothendieck group:
\begin{equation}
\label{e:Mm} \begin{aligned}
{[\std(\delta)]}&=\sum_{\gamma \in \ol\scB}
m(\gamma,\delta)[\irr(\gamma)];\\ 
[\irr(\delta)]&=\sum_{\gamma \in \ol\scB}
M(\gamma,\delta)[\std(\gamma)].
\end{aligned} 
\end{equation} 

For the linear groups considered in Section \ref{s:notation} the
above matrices are computable by Vogan's algorithm
\cite{ic3}.  For any two-fold cover of such a group they are computable by
the main results of \cite[Part I]{rt3}.

\begin{definitionplain}
\label{d:vogandual}
Suppose $G$ and $G'$ are real reductive groups in Harish-Chandra's class,
with blocks $\b$ and $\b'$ having regular infinitesimal
character. Write $\ol\scB$ and $\ol\scB'$ for the sets
parametrizing $\b$ and $\b'$. 

Suppose there is a bijection
\begin{align*}
\Phi \; : \; \ol\scB &\lra \ol\scB' \\
\delta &\lra \delta'
\end{align*}
and a parity function
\[
\epsilon \; : \;\ol\scB \times\ol\scB' \lra \{\pm 1\}
\]
satisfying
\[
\epsilon(\gamma, \delta) \epsilon(\delta, \eta) = \epsilon(\gamma, \eta),
\qquad \epsilon(\gamma, \gamma) = 1.
\]
Then we say that $\b$ is  \Vogan dual to $\b'$ (with respect to $\Phi$)
if for all $\delta, \gamma \in \ol\scB$,
\begin{subequations}
\renewcommand{\theequation}{\theparentequation)(\alph{equation}}
\begin{align}
m(\delta,\gamma)&=\epsilon(\delta,\gamma)M(\gamma',\delta'); 
\text{ or equivalently,} \\
M(\delta,\gamma)&=\epsilon(\delta,\gamma)m(\gamma',\delta').
\end{align}
\end{subequations}
We say $\b'$ is  \Vogan {\em dual} to $\b$ 
if it is \Vogan dual to $\b$ with respect to some bijection $\Phi$.
\end{definitionplain}

\begin{remark}
\label{r:hecke}
Typically the duality of two blocks $\b$ and $\b'$ is deduced from a
stronger kind of duality.  For instance, consider the setting of
\cite{ic4}, and fix a block $\b$ at infinitesimal character $\lam$.
Then the $\bbZ[q,q^{-1}]$ span $\caM$ of the elements of $\b$ is
naturally a module for the Hecke algebra of the integral Weyl group of
$\lam$.  The paper \cite{ic4} constructs a block $\b'$ at
infinitesimal character $\lam'$ such that the integral Weyl group of
$\lam'$ is isomorphic as a Coxeter group to the integral Weyl group of
$\lam$.  Denote the associated Hecke algebra $\caH$, and write the
$\caH$ modules corresponding to $\b$ and $\b'$ as $\caM$ and $\caM'$.
Each module has a natural basis corresponding to the standard modules
(Section \ref{s:regularCharacters}) attached to block elements.  The
$\Z[q,q^{-1}]$-linear dual $\caM^*$ of $\caM$ also has a natural
structure of an $\caH$ module, and it is equipped with the basis dual
to the standard module basis for $\caM$.  Roughly speaking, Vogan
defines a bijection from this basis of $\caM^*$ to the basis of
standard modules for $\caM'$ (up to sign), and proves that this induces an
isomorphism of $\caH$ modules.  His proof of the Kazhdan-Lusztig
conjecture immediately implies that $\b$ and $\b'$ are dual in the
sense of Definition \ref{d:vogandual}.  In this way the character multiplicity
duality statement of \cite[Theorem
  1.15]{ic4} is deduced from the duality of $\caH$ modules in 
\cite[Proposition 13.12]{ic4}.

For the simply laced groups we consider below, an (extended) Hecke algebra
formalism is provided by \cite{rt3}.  We will ultimately prove Theorem
\ref{t:mainintro} by proving a duality of certain modules for an
extended Hecke algebra in the proof of Theorem \ref{t:gendual} below.
\end{remark}

\begin{remark}
\label{r:length}
In cases in which the matrices $M(\gamma,\delta)$ are computable, the
computation depends in an essential way on a length function
$l:\ol\scB\rightarrow \mathbb N$.
In the setting below, the correct notion in our setting is the extended
integral length of \cite{rt3}.  (When $G$ is
linear, this reduces to the notion of integral length defined in
\cite{green}.)  Then the parity function appearing in Definition
\ref{d:vogandual} may be taken to be
\[
\epsilon(\gamma, \delta) = (-1)^{l(\gamma) + l(\delta)}.
\]
\end{remark}

% \sec{Passing from $\h$ to $\h^*$}
%(This isn't really needed, and is wrong for $G_2$...)

% \label{s:passing}

% Fix a complex reductive Lie algebra $\g$ and let $\h$ be
% a Cartan subalgebra. Let $\Delta=\Deltagh\subset\h^*$ be 
% the set of roots of $\h$ in $\g$.
% Let $\h_d=\h\cap \g_d$ where $\g_d$ is the derived algebra of $\g$. 
% Let $\ch\Delta\subset\h_d$ be the 
% set of coroots \cite[Definition 1, Chapter VI]{bourbaki}. 
% Let $(\cdot,\cdot)$ be a nondegenerate 
% Weyl group invariant bilinear form on $\h$ satisfying
% $(\ch\alpha,\ch\alpha)=2$ for all long roots $\alpha$. 
% Such a form exists and is
% uniquely determined on $\h_d$.

% Write
% $\langle \cdot , \cdot \rangle$ for the natural pairing
% $\h^*\times\h\rightarrow\C$. 
% The form $(\cdot,\cdot)$ induces an
% isomorphism $$ \Phi:\h^*\simeq\h $$ 
% defined by

% \begin{equation}
% \label{e:Phi}
% (\Phi(\lambda),X)=\bra\lambda X\quad(\lambda\in\h^*,X\in\h).
% \end{equation}

% \begin{subequations}
% \renewcommand{\theequation}{\theparentequation)(\alph{equation}}
% It is easy to check that for $\alpha\in\Delta,\lambda\in\h^*$,
% \begin{equation}
% \Phi(\alpha)=\frac
% 2{(\ch\alpha,\ch\alpha)}\ch\alpha
% \end{equation}
% and 
% \begin{equation}
% \langle\alpha,\Phi(\lambda)\rangle
% =
% \frac2{(\ch\alpha,\ch\alpha)}\langle\lambda,\ch\alpha\rangle.
% \end{equation}
% See \cite[Section 3]{shimura}. 
% If $\g$ is simply laced, these formulas
% become
% \begin{equation}
% \Phi(\alpha)=\ch\alpha
% \end{equation}
% and 
% \begin{equation}
% \label{e:relation}
% \bra\alpha{\Phi(\lambda)}
% =\brach\lambda\alpha.
% \end{equation}
% \end{subequations}

\sec{Regular Characters}
\label{s:regularCharacters}

In this section $G$ is a real reductive linear group (Section \ref{s:notation})
and $\tG$ is a two-fold cover of $G$.
Since $\tG$ is a central extension of $G$, a Cartan subgroup $\tH$ of $\tG$
is the inverse image of a Cartan subgroup $H$ of $G$.
It is often the case that $\tH$ is not abelian.
Note that
\begin{equation}
\label{e:zh}
\wt{H^0}\subset Z(\tH)
\end{equation}
since  the map $(g,h)\rightarrow ghg\inv h\inv$ is a
continuous map from $\tH\times\tH$ to $\pm1$.
In particular $|\tH/Z(\tH)|$ is finite.

Equivalence classes of
irreducible genuine representations of $\tH$ are parametrized by
genuine characters of $Z(\tH)$ according to the next lemma.

\begin{lemma}
\label{l:repT}
Write $\Pi(Z(\tH))$ and
$\Pi(\tH)$ for equivalence classes of
irreducible genuine representations of $Z(\tH)$ and
$\tH$, respectively, and let $n=|\tH/Z(\tH)|^{\frac12}$.
For every $\chi\in\Pi(Z(\tH))$ there is a unique representation $\pi=\pi(\chi)\in\Pi(\tH)$ 
for which $\pi|_{Z(\tH)}$ is a multiple
of $\chi$.
The map
$\chi\rightarrow\pi(\chi)$ is a bijection between
$\Pi(Z(\tH))$ and $\Pi(\tH)$.
The dimension of $\pi(\chi)$ is $n$, and
$Ind_{\wt{Z(H)}}^{\tH}(\chi)=n\pi$.
\end{lemma}

The proof is elementary; for example see \cite[Proposition 2.2]{shimura}.
The lemma shows that an irreducible representation of $\tH$ is
determined by a character of $Z(\tH)$. 
The fact that this is  smaller than $\tH$ makes the representation
theory of $\tG$ in many ways  simpler than that of $G$.
In fact for an admissible cover of a simply laced group we have the
following important result.

\begin{proposition}[\cite{aherb}, Proposition 4.7]
\label{p:ZH}
Suppose $G$ is simply laced, $\tG$ is an admissible two-fold cover of
$G$, and $H$ is a Cartan subgroup of $G$. 
Then
\begin{equation}
Z(\tH)=Z(\tG)\wt{H^0}.
\end{equation}
In particular a genuine character of $Z(\tH)$ is determined by its
restriction to $Z(\tG)$ and its differential.  
\end{proposition}

\begin{definitionplain}[{{\cite[Definition 2.2]{vogan:unit}}}]
\label{d:reg}
Suppose $G$ is a real reductive group in Harish-Chandra's class.
A regular character of $G$ is a triple
\begin{equation}
\label{e:regchar}
\gamma=(H,\Gamma,\lambda)
\end{equation}
consisting of a $\theta$-stable Cartan subgroup $H$,
an irreducible representation $\Gamma$ of $H$, and
$\lambda\in\h^*$, satisfying the following conditions. 
Let $\Delta=\Deltagh$,
let $\Delta^i$ be the imaginary roots, and $\Delta^{i,c}$ the
imaginary compact roots.
The first condition is
\begin{subequations}
\renewcommand{\theequation}{\theparentequation)(\alph{equation}}
\begin{equation}
\label{e:regchar1}
\brach\lambda\alpha \in \mathbb{R}^\times
\text{ for all }\alpha\in\Delta^i.
\end{equation}
Let
\begin{equation}
\label{e:regchar2}
\rho_i(\lambda)=\frac12\sum_{\substack{\alpha\in\Delta^i\\ \langle\lambda,\ch\alpha\rangle>0}}\alpha,
\quad\quad
\rho_{i,c}(\lambda)=\frac12\sum_{\substack{\alpha\in\Delta^{i,c}\\ \langle\lambda,\ch\alpha\rangle>0}}\alpha.
\end{equation}
The second condition is
\begin{equation}
\label{e:regchar4}
d\Gamma=\lambda+\rho_i(\lambda)-2\rho_{i,c}(\lambda).
\end{equation}
\end{subequations}
Finally we assume 
\begin{equation}
\label{e:reg}
\langle \lam, \ch \alpha \rangle \neq 0 \text{ for all } \alpha \in \Delta.
\end{equation}
We say $\gamma$ is {\it genuine} if $\Gamma$ is genuine.
\end{definitionplain}

Of course if $G$ is linear then $H$ is abelian, and $\Gamma$ is a
character of $H$. If $G$ is nonlinear (and $\gamma$ is genuine) by Lemma \ref{l:repT} we could
replace the irreducible representation $\Gamma$ of $H$ with a
character of $Z(H)$. 

Write $H=TA$ as usual, and let $M$ be the centralizer of $A$ in $G$.  The conditions on
$\gamma$ imply that there is a unique relative discrete series
representation of $M$, denoted
$\pi_M$, with Harish-Chandra parameter $\lam$,
whose lowest $M\cap K$-type has $\Gamma$ as a highest weight.
Equivalently $\pi_M$ is determined by $\lambda$ and its central
character, which is $\Gamma$ restricted to the center of $M$.
Define a parabolic subgroup $MN$ by
requiring that the real part of $\lam$ restricted to $\fra_\o$ be
(strictly) positive on the roots of $\fra$ in $\frn$.  Then set
\[
\std(\gamma) = \Ind_{MN}^G(\pi_M \otimes 1 \! \! 1).
\]
By the choice of $N$ and \eqref{e:reg}, $\std(\gamma)$ has a unique
irreducible quotient which we denote $\irr(\gamma)$.
The central character of $\pi(\gamma)$ is
$\Gamma$ restricted to $Z(G)$, which we refer to as the central
character of $\gamma$.

The maximal compact subgroup $K$ acts in the obvious way on the set of
regular characters, and the equivalence class of $\irr(\gamma)$ only
depends on the $K$ orbit of $\gamma$.
Write $cl(\gamma)$ for the $K$-orbit of $\gamma$, and
$\gamma\sim\gamma'$ if $cl(\gamma)=cl(\gamma')$.

Fix a maximal ideal $I$ in the center of the universal
enveloping algebra $U(\g)$ of $\g$.  We call $I$ an infinitesimal
character for $\g$, and as usual say that a $U(\g)$ module has
infinitesimal character $I$ if it is annihilated by $I$.
For $\h$ a Cartan subalgebra  and $\lambda\in\h^*$ let
$I_\lambda$ be the infinitesimal character for $\g$ determined by
$\lambda$ via the Harish-Chandra homomorphism.
We say $I_\lambda$ is regular if $\lambda$ is regular.
If
$\gamma=(H,\Gamma,\lambda)$ is a regular character then
$\pi(\gamma)$ has infinitesimal character $I_\lambda$. 

Fix a regular infinitesimal character $I$.
Let $\caS_I$ be the set of 
$K$-orbits of regular characters $\gamma=(H,\Gamma,\lambda)$ with
$I_\lambda=I$.
In this setting the Langlands classification takes the following form.
See \cite[Section 2]{vogan:unit}, for example.

\begin{theorem}
\label{t:lang}
Suppose $G$ is a real reductive group in Harish-Chandra's class.
Fix a regular infinitesimal character $I$ for $\g$.
The map $\gamma\rightarrow
\overline\pi(\gamma)$ is a bijection from $\caS_I$ to the set of
equivalence classes of irreducible admissible representations of $G$ with 
infinitesimal
character $I$.
\end{theorem}

Proposition \ref{p:ZH} implies the following important rigidity result
for genuine regular characters of admissible two-fold covers. 

\begin{proposition}
\label{p:gammarigid}
Let $\tG$ be an admissible two-fold cover of a simply laced real
reductive linear group.
Suppose $\gamma=(\tH,\Gamma,\lambda)$ is a genuine regular character of
$\tG$. 
Then $\gamma$ is determined by
$\lambda$ and the restriction of $\Gamma$ to $Z(\tG)$.
\end{proposition}

This follows immediately from Proposition \ref{p:ZH}
since the differential of  $\Gamma$ is determined
by $\lambda$ by \eqref{e:regchar4}. Consequently $\tG$ typically has
few genuine irreducible representations.

\begin{exampleplain}
\label{ex:verysimple}  
Suppose $\tG$ is split and fix a genuine 
central character and an infinitesimal character for $\tG$.
Then there is precisely
one minimal principal series representation with the given
infinitesimal  and central characters.
This is very different from a linear group; for example if $G$ is
linear and semisimple of rank $n$ it has $2^n/|Z(G)|$ minimal
principal series representation with given infinitesimal and central
character. 
\end{exampleplain}

\begin{exampleplain}
\label{ex:verysimple2}
This example is dual to the preceding one.
Suppose $\tG$ is an admissible cover of
a real form $G$ of a simply connected, semisimple
complex group. Then $\tG$ has at most 
one discrete series representation with given infinitesimal and 
central character. 
Again this is very different from the linear case. 
\end{exampleplain}

\begin{definitionplain}
\label{d:blockregchar}
We say two genuine regular characters $\gamma$ and
$\delta$ are block equivalent if  $\irr(\gamma)$ and
$\irr(\delta)$ are block equivalent as in 
Section \ref{s:dual}.
A block of
regular characters is an equivalence class for this relation.
\end{definitionplain}

We write $\scB$ for a block of regular characters. Clearly $K$ acts on
$\scB$, and we set $\ol\scB=\scB/K$. If $\gamma$ is a regular character
write $\scB(\gamma)$ for the block of regular characters containing $\gamma$, and
$\ol\scB(\gamma)=\scB(\gamma)/K$.
Fix $\gamma_0$, and let $\ol\scB=\ol\scB(\gamma_0)$ and
$\b(\irr(\gamma_0))$ be the corresponding blocks of regular characters
and representations, respectively. 
Then the map
$\gamma\rightarrow\irr(\gamma)$ factors to $\ol\scB$, 
and the map
\begin{equation}
\ol\scB\ni\gamma\rightarrow \irr(\gamma)\in\b
\end{equation}
is a bijection.

\medskip

To conclude, we note that Proposition \ref{p:gammarigid} has a
geometric interpretation.  Revert for a moment to the general setting
of $G$ in Harish-Chandra's class.  Let $\frB$ be the variety of Borel
subalgebras of $\g$.  Given a regular character
$\gamma=(H,\Gamma,\lambda)$ of $G$, let $\mathfrak b$ be the Borel
subalgebra of $\g$ containing $\h$ making $\lambda$ dominant.  This
defines a map from $K$ orbits of regular characters for $G$ to $K_\C$
orbits on $\frB$ (see \cite[Proposition 2.2(a)--(c)]{ic3}).  
Using Theorem \ref{t:lang}, we may interpret this
as a map, which we denote $\supp$, from irreducible representations of
$G$ (with regular infinitesimal character) to $K_\C$ orbits on $\frB$.
Although we do not need it, we remark that
$\supp(\pi)$ is indeed (dense in) the support of a suitable localization
of $\pi$; see \cite[Section 5]{ic3}.

\begin{proposition}
\label{p:geom}
Suppose $\tG$ is an admissible cover of $G$, a simply laced
linear real reductive group.  Let $\caB$ be a block of genuine 
representations of $\tG$ with regular infinitesimal character.
Then $\supp$ is injective, and hence $\caB \hookrightarrow
K_\C\bs \frB$.
\end{proposition}

\begin{proof}
Note that since $\tK$ is a central extension of $K$, the orbits
of $(\tK)_\C$ and $K_\C$ on $\frB$ coincide.  Since any two representation in
$\caB$ have the same central and infinitesimal character, the
proposition follows from  Proposition \ref{p:gammarigid}.
\end{proof}

\sec{Bigradings}
\label{s:bigradings}
We state versions of some of the results of \cite{ic4} in our setting.
Throughout this section let $G$ be a real reductive linear group and 
$\tG$ an admissible two-fold cover of $G$  (Section
\ref{s:notation} and Definition \ref{d:admissible}).

Fix a $\theta$-stable Cartan subgroup $H$ of $G$. Let
$\Delta = \Deltagh$ and set
\begin{equation}
\label{e:integral}
\Delta(\lambda)=\{\alpha\in\Delta\,|\, \brach\lambda\alpha\in\Z\},
\end{equation}
the set of integral roots defined by $\lambda$.
Define
\begin{equation}
m(\alpha)=
\begin{cases}
  2&\alpha\text{ metaplectic}\\
1&\text{otherwise}.
\end{cases}
\end{equation}
See \cite[Definition 6.5]{rt3}.
Define 
% In the notation of Section \ref{s:passing},
% \begin{equation}
% \label{e:halfintegral}
% \Delta_{\frac12}(\lambda)=\{\alpha\in\Delta\,|\,
% \langle\lambda,\Phi(\alpha)\rangle\in\frac12\Z\}.
% \end{equation}
\begin{equation}
\label{e:halfintegral}
\Delta_{\frac12}(\lambda)=\{\alpha\in\Delta\,|\,
\langle\lambda,m(\alpha)\ch\alpha\rangle\in\Z\}.
\end{equation}
A short argument shows that this is always
a $\theta$-stable root system.
Note that if $G$ is simply laced then
$\Delta_{\frac12}(\lambda)=\Delta(2\lambda)$.
This also holds for $G_2$, since in this case all roots are
metaplectic. 
Let $W_{\frac12}(\lambda)=W(\Delta_\frac12(\lambda))$.
\begin{definitionplain}
\label{d:halfint}
We say $\lambda$ is half-integral if $\Delta_{\frac12}(\lambda)=\Delta$.
\end{definitionplain}

Let $\Delta^r$ be the set of real roots, i.e. those roots for which
$\theta(\alpha)=-\alpha$. Let
\begin{equation}
\begin{aligned}
\Delta^r(\lambda)&=\Delta^r\cap \Delta(\lambda)\\
\Delta^r_{\frac12}(\lambda)&=\Delta^r\cap \Delta_{\frac12}(\lambda).
\end{aligned}
\end{equation}
Let $\Delta^i$ be the imaginary roots, and
define $\Delta^i(\lambda)$ and 
$\Delta^i_\frac12(\lambda)$ similarly.

Now fix a genuine regular
character $\gamma=(H,\Gamma,\lambda)$ (Definition \ref{d:reg}). 
For $\alpha\in \Delta_{\frac12}^r(\lambda)$ let
$m_\alpha\in\tG$ be defined as in 
\cite[Section 5]{rt3}.  Thus $m_\alpha$ 
is an inverse image of the corresponding
element of $G$ of \cite[4.3.6]{green}.  Recalling the notion of
metaplectic roots (Definition \ref{d:meta}), 
it is well-known that
\begin{equation}
\label{e:metalpha}
\alpha\text{ is metaplectic  if and only if } m_\alpha \text{
has order 4.}
\end{equation}
We say $\alpha$ satisfies the
parity condition with respect to $\gamma$ if 
\begin{equation}
\label{e:parity}
\text{the eigenvalues of }
\Gamma(m_\alpha)\text{ are of the form }-\epsilon_\alpha e^{\pm\pi i\brach\lambda\alpha}
\end{equation}
where $\epsilon_\alpha=\pm1$ as in \cite[Definition 8.3.11]{green}.
% (cf.~Lemma \ref{l:epsilon}).
Note that if $\alpha\in\Delta(\lambda)$ then
$e^{\pm\pi i\brach\lambda\alpha}=(-1)^{\brach\lambda\alpha}$, and this definition agrees with
\cite[Definition 8.3.11]{green}.

Let
\begin{subequations}
\label{e:psi}
\renewcommand{\theequation}{\theparentequation)(\alph{equation}}
\begin{align}
\Delta_{\frac12}^{r,+}(\gamma)&=\{\alpha\in\Delta_{\frac12}^r(\lambda)\,|\, \alpha\text{ does not satisfy
  the parity condition}\}
\label{e:psir+}
\\
\Delta_{\frac12}^{r,-}(\gamma)&=\{\alpha\in\Delta_{\frac12}^r(\lambda)\,|\,
\alpha\text{ satisfies the parity condition}\}.
\label{e:psir-}
\end{align}
Also define
\begin{align}
\Delta_{\frac12}^{i,+}(\gamma)&=\{\alpha\in\Delta_{\frac12}^i(\lambda)\,|\, \alpha\text{ is
  compact}\}
\label{e:psii+}
\\
\Delta_{\frac12}^{i,-}(\gamma)&=\{\alpha\in\Delta_{\frac12}^i(\lambda)\,|\, \alpha\text{ is
  noncompact}\}
\label{e:psii-}.
\end{align}
Finally let
\begin{equation}
\begin{aligned}
\label{e:Deltair}
\Delta^{i,\pm}(\gamma)&=\Delta^{i,\pm}_{\frac12}(\gamma)\cap \Delta(\lambda)\\
\Delta^{r,\pm}(\gamma)&=\Delta^{r,\pm}_{\frac12}(\gamma)\cap \Delta(\lambda).\\
\end{aligned}
\end{equation}
These are the usual integral imaginary compact and noncompact roots,
and the integral real roots which satisfy (or do not satisfy) the
parity condition.
\end{subequations}

It is a remarkable fact that 
$\Delta^{i,\pm}_{\frac12}(\gamma)$
and
$\Delta^{r,\pm}_{\frac12}(\gamma)$
 only depend on $\theta$ and $\lambda$, as the next proposition shows.
This is very different from the linear case.

% It is worth keeping in mind that the sets
% $\Delta^{i,\pm}_{\frac12}(\gamma)$ only depend on $\lambda$ and
% $\theta$ (not just $\theta$ restricted to $\h$), and we write
% \begin{equation}
% \Delta^{i,\pm}_{\frac12}(\gamma)=
% \Delta^{i,\pm}_{\frac12}(\lambda,\theta)
% \end{equation}
% accordingly.
% Thus we have
% $$
% \Delta_{\frac12}^{i,\pm}(\lambda,\theta)\subset
% \Delta_{\frac12}^i(\lambda,\theta|_\h)\subset \Delta_{\frac12}(\lambda).
% $$

% The corresponding fact, that $\Delta^{r,\pm}_{\frac12}(\gamma)$ only
% depends on $\lambda$ and $\theta$ is not apriori true (and is
% false in the linear case). A basic result is that in fact a much
% stronger result holds: both
% $\Delta^{i,\pm}_{\frac12}(\gamma)$  and $\Delta^{r,\pm}_{\frac12}(\gamma)$ only depend
% on $\lambda$ and $\theta|_\h$.

\begin{proposition}[Corollaries 6.9 and 6.10 of {\cite{rt3}}]
\label{p:gradings}
Let $\tG$ be an admissible cover of a simply laced, real reductive
linear group $G$.  Let $\gamma = (H,\Gamma, \lambda)$ be a genuine
regular character for $\tG$. Then
\begin{subequations}
\label{e:gradings}
\renewcommand{\theequation}{\theparentequation)(\alph{equation}}
\begin{align}
\Delta_{\frac12}^{r,+}(\gamma)=\Delta^r(\lambda)\label{e:gradingsa}\\
\Delta_{\frac12}^{i,+}(\gamma)=\Delta^i(\lambda)\label{e:gradingsb}.
\end{align}
\end{subequations}
In other words if $\alpha$ is a real, half-integral root then it fails
to satisfy the parity condition if and only if it is
integral. Similarly an imaginary root  (which is necessarily
half-integral) is compact if and only if it is integral.
\end{proposition}

\begin{proof}
Suppose $\alpha\in\Delta_{\frac12}(\lambda)$ is a real root.
According to Definition \ref{d:admissible}, $\alpha$ is metaplectic, so by 
\eqref{e:metalpha} $m_\alpha$ has order $4$.
Therefore, since $\Gamma$ is genuine,
$\Gamma(m_\alpha)$ has eigenvalues $\pm i$. 
On the other hand 
$-\epsilon_\alpha e^{\pi i\brach\lambda\alpha}=\pm 1$ if
$\alpha\in\Delta(\lambda)$, or $\pm i$ if $\alpha\in \Delta_{\frac12}(\lambda)\; \backslash \;
\Delta(\lambda)$. 
Therefore the parity condition \eqref{e:parity} fails if 
$\alpha\in\Delta(\lambda)$, and holds otherwise. This proves (a).
For (b) we may reduce to the case of a simply laced group, in which case
it is an immediate consequence of the following lemma.
\end{proof}

\begin{lemma}
\label{l:wtexp}
Let $\tG$ be an admissible cover of a simply laced, real reductive
linear group $G$.  
Suppose $H$ is a $\theta$-stable Cartan subgroup of $G$, with inverse
image $\tH$, and let $\wt{\exp}:\h_0\rightarrow \tH$ be the exponential map.

\smallskip

\noindent (a) Suppose $\alpha$ is an imaginary root. Then
\begin{equation}
\wt{\exp}(2\pi i\ch\alpha)=
\begin{cases}
1&\alpha\text{ compact}\\  
-1&\alpha\text{ noncompact}.
\end{cases}
\end{equation}

\noindent (b) Let $\alpha$ be a complex root. Then
\begin{equation}
\wt{\exp}(2\pi i(\ch\alpha+\theta\ch\alpha))=
\begin{cases}
1&\langle \alpha,\theta\ch\alpha\rangle=0\\
-1&\langle \alpha,\theta\ch\alpha\rangle\ne 0.
\end{cases}
\end{equation}
\end{lemma}

\begin{proof}
First assume $\alpha$ is imaginary.  Recall
the group $M_\alpha$ introduced just
before Definition \ref{d:meta}.  The proposition reduces to the
case of $G = M_\alpha$, which is locally isomorphic to
$\SL(2,\R)$ or $\SU(2)$.  If $\alpha$ is compact then $M_\alpha$ and
the identity component of $\wt{M_\alpha}$ are either isomorphic to
$\SU(2)$ or $\SO(3)$. The fact that $\wt{\exp}(2\pi i\ch\alpha)=1$
reduces to the case of a linear group.

If $\alpha$ is
noncompact then $M_\alpha$ is locally isomorphic to $\SL(2,\R)$ and  
by Lemma \ref{l:metaplecticRoots}(1) 
and the admissibility hypothesis
$\wt{M_\alpha}$ is the nonlinear cover $\wt{\SL}(2,\R)$.
It is well-known that
$\wt{\exp}(\pi i\ch\alpha)$ has order $4$ (cf.~\eqref{e:metalpha}),
so 
$\wt{\exp}(2\pi i\ch\alpha)=-1$. 

Now suppose  $\alpha$ is complex.
If $\langle\alpha,\theta(\ch\alpha)\rangle=-1$ then
$\beta=\alpha+\theta(\alpha)$ is a noncompact imaginary root,  
$\ch\beta=\ch\alpha+\theta(\ch\alpha)$, and we are reduced to the
previous case.

If $\langle\alpha,\theta(\ch\alpha)\rangle=1$, then
$\beta=\alpha-\theta(\alpha)$ is a real root. 
By taking a Cayley transform by $\beta$ we obtain a new Cartan
subgroup with an imaginary noncompact root, and again reduce to
the previous case (see Section \ref{s:cayley}). 

Finally suppose $\langle\alpha,\theta(\ch\alpha)\rangle=0$.  
Let $\mathfrak l_\alpha$ be the subalgebra of $\g$ 
generated by root vectors $X_{\pm \alpha},X_{\pm\theta(\alpha)}$.  Let
$L_\alpha$ be the subgroup of $G$ with Lie algebra $\mathfrak
l_\alpha\cap \g_0$.  
The assumption $\langle\alpha,\theta(\ch\alpha)\rangle=0$ implies 
$L_\alpha$ is locally isomorphic to $\SL(2,\C)$, which has no
nontrivial cover.
Then $\wt\exp(2\pi i(\ch\alpha+\theta\ch\alpha)=1$ by the
corresponding fact for linear groups.
\end{proof}

We recall some definitions from \cite[Section 3]{ic4}. Suppose $\Delta$ is
a root system. A {\it grading} of $\Delta$ is a map
$\epsilon:\Delta\rightarrow\pm1$, such that
$\epsilon(\alpha)=\epsilon(-\alpha)$ and
$\epsilon(\alpha+\beta)=\epsilon(\alpha)\epsilon(\beta)$ whenever
$\alpha,\beta$ and $\alpha+\beta\in \Delta$.  A {\it cograding} of $\Delta$ is a
map $\ch\delta:\Delta\rightarrow\pm1$ such that the dual map $\delta:\ch
\Delta\rightarrow\pm1$ (defined by $\delta(\ch\alpha)=\ch\delta(\alpha)$)
is a grading of the dual root system $\ch \Delta$.
We identify a grading $\epsilon$ with its kernel $\epsilon\inv(1)$,
and similarly for a cograding.

A {\it bigrading} of $\Delta$ is a triple $g=(\theta,\epsilon,\ch\delta)$ where
$\theta$ is an involution of $\Delta$, $\epsilon$ is a grading of
$\Delta^i=\Delta^\theta$, and $\ch\delta$ is a cograding of $\Delta^r=\Delta^{-\theta}$.
(This is a {\it weak bigrading} of \cite[Definition 3.22]{ic4}.)
We write $g=(\Delta,\theta,\epsilon,\ch\delta)$ to keep track of $\Delta$.
% We will not make any use of strong bigradings.
Alternatively we write $g=(\theta,\Delta^{i,+},\Delta^{r,+})$ where 
$\Delta^{i,+}\subset \Delta^i$ is the kernel of $\epsilon$ and $\Delta^{r,+}\subset
\Delta^r$ is the kernel of $\ch\delta$.
The {\it dual bigrading} of $g$ is the bigrading $\ch
g=(\ch\Delta,-\ch\theta,\delta,\ch\epsilon)$.

\medskip
It is easy to see that $\Delta_{\frac12}^{i,+}(\gamma)$ is a grading of
$\Delta_{\frac12}^i(\lambda)$, and
by \eqref{e:gradingsa}
$\Delta_{\frac12}^{r,+}(\gamma)$ 
is a grading of
$\Delta_{\frac12}^r(\lambda)$.
In the simply laced case gradings and cogradings coincide.
We restrict to this case for the remainder of this section.

\begin{definitionplain}
\label{d:big}
Assume $G$ is simply laced.
The bigrading of $\Delta_{\frac12}(\lambda)$ defined by a genuine
regular character $\gamma$ of $\tG$ is
$$
\ghalf(\gamma)=(\Delta_{\frac12}(\gamma),\theta,\Delta_{\frac12}^{i,+}(\gamma),
\Delta_{\frac12}^{r,+}(\gamma)).
$$
\end{definitionplain}

\smallskip

\noindent 
With the obvious notation
$g(\gamma)=\ghalf(\gamma)\cap \Delta(\lambda)$ is a
bigrading of $\Delta(\lambda)$.

\begin{subequations}
\renewcommand{\theequation}{\theparentequation)(\alph{equation}}  
By Proposition \ref{p:gradings}
\begin{equation}
\label{e:bigrading}
\begin{aligned}
\ghalf(\gamma)&=(\Delta_{\frac12}(\lambda),\theta, \Delta^i(\lambda),\Delta^r(\lambda))\\
g(\gamma)&=(\Delta(\lambda),\theta, \Delta^i(\lambda),\Delta^r(\lambda))\\
\end{aligned}
\end{equation}
Identifying $\Delta$ and $\ch\Delta$, the 
dual bigrading to $\ghalf(\gamma)$ is
\begin{equation}
g^\vee_{\frac12}(\gamma)=(\Delta_{\frac12}(\lambda),-\theta,\Delta^r(\lambda),\Delta^i(\lambda)).
\end{equation}
\end{subequations}

% The proposition also implies
% \begin{equation}
% \label{e:same}
% \Delta^{i,+}_{\frac12}(\gamma)=
% \Delta^{i,+}(\gamma),\quad
% \Delta^{r,+}_{\frac12}(\gamma)=
% \Delta^{r,+}(\gamma).
% \end{equation}

% \begin{comment}
% \begin{remark}
% \label{r:defofbigs}
% In the two root length case, $\Delta_{\frac12}^{r,+}(\gamma)$ 
% is not necessarily a cograding of $\Delta_{\frac12}^r(\lambda)$.
% Thus $\ghalf(\gamma)$
% consists of a pair of gradings, rather than a grading and a
% cograding, as in
% \cite[Definition 3.22]{ic4}. 
% \end{remark}

% \begin{remark}
% \label{r:difference}
%  Parts (a) and (b) of the proposition imply that the bigrading of
%  $\Delta_{\frac12}(\lambda)$ is determined by $\theta|_\h$. This is
%  the key difference between linear and nonlinear groups; in the
%  linear case $g(\gamma)$ is not determined by $\theta|_\h$, or even
%  by $\theta$.  (In the case of a nonlinear group with two root
%  lengths the bigrading is not completely determined by $\theta|_\h$ either,
%  but even in this case it is much closer to being
%  true than for linear groups.)
% \end{remark}
% \end{comment}

\bigskip

The point of Definition \ref{d:big} is that is contains the information
necessary to prove a duality theorem for simply admissible covers of
simply laced groups.
To the reader familiar with \cite{ic4} (where 
strong bigradings are needed), it
is perhaps surprising that we can get away with so little.  We
offer a few comments explaining this.

Let $G$ be a real reductive linear group.
Suppose $\gamma=(H,\Gamma,\lambda)$ is a regular character of $G$
(Definition \ref{d:reg}).
Write $H=TA$ as usual, and let $M$ be the centralizer of $A$ in $G$.
Section 4 of \cite{ic4} associates a strong bigrading of the
integral root system $\Delta(\lam)$ to $\gamma$.
One component of the strong bigrading is
the imaginary cross-stabilizer $W^i(\gamma)$ (see Section
\ref{s:ca}). This group (denoted $W^i_1(\gamma)$ in \cite{ic4}) 
satisfies the following containment relations:
\begin{subequations}
\renewcommand{\theequation}{\theparentequation)(\alph{equation}}
\label{e:strongbig}
\begin{equation}
\label{e:strongbig1}
W(\Delta^{i,+}(\lam))\subset W^i(\gamma) \subset
\text{Norm}_{W(\Delta^i(\lam))}(\Delta^{i,+}(\lam)).
\end{equation}
The outer terms in \eqref{e:strongbig1} depend only on Lie algebra
data, while $W^i(\gamma)$ depends
on the disconnectedness of $G$.

Dually, a strong bigrading also includes the real cross stabilizer $W^r(\gamma)$
satisfying
\begin{equation}
\label{e:strongbig2}
W(\Delta^{r,+}(\gamma))\subset W^r(\gamma)\subset
\Norm_{W(\Delta^r(\lambda))}(\Delta^{r,+}(\lambda)).
\end{equation}
\end{subequations}

Now suppose  $\gamma=(H,\Gamma,\lambda)$ is a  genuine regular character
of an admissible cover $\tG$ of $G$, with $G$ simply laced.
Then Proposition \ref{p:gradings} implies that the outer
two terms in \eqref{e:strongbig1} and (b) are the same,
so the middle terms are  determined:
\begin{equation}
\label{e:w1s}
W^i(\gamma) = W(\Delta^i(\lam)), \quad
W^r(\gamma) = W(\Delta^r(\lam)).
\end{equation}
Thus the (weak) bigrading determines the strong bigrading in this setting.
Furthermore, the   only information contained in
$g_{\frac12}(\gamma)$  not already in $g(\gamma)$ (see \eqref{e:bigrading}) 
is the action of $\theta$ on the roots which are half-integral but not
integral.  See the proof of Proposition \ref{p:crossstab}.

\sec{Cross Action and Cayley Transforms}
\label{s:ca}

We start by defining the cross action of the Weyl group on genuine
regular characters.
Throughout this section let $G$ be a real reductive linear group and 
$\tG$ an admissible two-fold cover of $G$  (Section
\ref{s:notation} and Definition \ref{d:admissible}).
Starting in Section \ref{s:cross} we assume
that $G$ is simply laced.

\subsection{The family $\caF$}
\label{s:F}

Fix a Cartan subalgebra $\h$ of $\g$ and let $\Delta=\Delta(\g,\h)$,
$W=W(\Delta)$ and let $R$ be the root lattice. Fix a set $\Delta^+$ of positive roots.
As in
Section \ref{s:regularCharacters} fix an infinitesimal character $I$,
which determines a dominant element $\lambda_\o\in\h^*$.
Recall (as in \eqref{e:integral})  $\Delta(\lambda_0)=\{\alpha\,|\,
\langle\lambda_0,\ch\alpha\rangle\in\Z\}$, and let
$W(\lambda_0)=W(\Delta(\lambda_0))=\{w\,|\, w\lambda_0-\lambda_0\in R\}$.

The map $w\rightarrow w\lambda_0$ induces
an isomorphism
\begin{equation}
\label{e:isomorphism}
W/W(\lambda_0)\simeq(W\lambda_0+R)/R.
\end{equation}
Let $\caF_0$ be a set of representatives of $(W\lambda_0+R)/R$ consisting
of dominant elements. For any $w\in W$ let $\lambda_w$ be
the element of $\caF_0$ corresponding to $w$ under the isomorphism
\eqref{e:isomorphism}. Thus 
\begin{subequations}
\label{e:props}
\renewcommand{\theequation}{\theparentequation)(\alph{equation}}  
\begin{align}
\lambda_w&\text{ is $\Delta^+$-dominant},\\
\lambda_x=\lambda_y&\text{ if and only if }y=xw\text{ with }w\in W(\lambda_0),\\
\lambda_w-w\lambda_0&\in R.
\end{align}
\end{subequations}
Let $\caF=W\caF_0$. Note that $|\caF|=|W||W/W(\lambda_0)|$.

\begin{definitionplain}
\label{d:Waction}
For $\lambda\in\caF$ and $w\in W$ define $w\cdot\lambda$ to be the
unique element of $\caF$ satisfying: $w\cdot\lambda$ is in the same
Weyl chamber as $w\lambda$, and $w\cdot\lambda-\lambda\in R$.
\end{definitionplain}

\begin{lemma}
\label{l:crossF}
This is well-defined, and is an action of $W$ on $\caF$.  
Furthermore for any $\lambda\in \caF$, $w\in W$,
 $w\cdot\lambda=w\lambda$ if and only
if $w\in W(\lambda)$.
\end{lemma}

\begin{proof}
This follows readily from an abstract argument, but perhaps it is
useful to make it explicit.

Given $\lambda$ choose $x,y\in W$ so that $\lambda=x\lambda_y$. We
claim for any $w\in W$
\begin{equation}
w\cdot(x\lambda_y)=wx\lambda_{x\inv w\inv xy}.
\end{equation}
Let $v=x\inv w\inv xy$. 
It is clear that $wx\lambda_v$ is in the same Weyl chamber as
$w\cdot(x\lambda_y)$, i.e. the Weyl chamber of $wx\lambda_y$.
We also have to show $wx\lambda_v-w\cdot(x\lambda_y)\in R$, which
amounts to
\begin{equation}
wx\lambda_v-x\lambda_y\in R.
\end{equation}
By \eqref{e:props}(c) write 
$\lambda_v=v\lambda_0+\tau$ and $\lambda_y=y\lambda_0+\mu$ with
$\tau,\mu\in R$.
We have to show
\begin{equation}
wx(v\lambda_0+\tau)-x(y\lambda_0+\mu)\in R.
\end{equation}
This is clear since $wxv=xy$ and $R$ is $W$-invariant.

Uniqueness is immediate, and the fact that this is an action is
straightforward.  The final assertion is also clear.
\end{proof}

\subsection{Cross Action}
\label{s:cross}

For the remainder of  section \ref{s:ca} we assume $G$ is simply
laced.

Fix a family $\caF$ as in the preceding section.
Now suppose $H$ is a Cartan subgroup of $\tG$ with complexified Lie
algebra $\h$, and suppose $\gamma=(H,\Gamma,\lambda)$ is a genuine regular character
with $\lambda\in\caF$.   Recall (Section \ref{s:regularCharacters})
the central character of $\gamma$ is
the restriction of $\Gamma$ to $Z(\tG)$.

\begin{definitionplain}
\label{d:cross}
For $w\in W$ define 
\begin{equation}
w\times\gamma=(H,\Gamma',w\cdot\lambda)
\end{equation}
where  $(H,\Gamma',w\cdot\lambda)$ is the unique genuine regular
character of this form with the same central character as that of
$\gamma$. 
\end{definitionplain}

\begin{lemma}
\label{l:cross}
The regular character $(H,\Gamma',w\cdot\lambda)$ of Definition
\ref{d:cross}
exists and is unique. This defines an action of $W$ on the set of
genuine regular characters $(H,\Gamma,\lambda)$ with $\lambda\in\caF$. 
Finally $\gamma$ and $w\times\gamma$ have the same infinitesimal
character if and only if $w\in W(\lambda)$, in which case 
Definition \ref{d:cross} agrees with  \cite[Definition 8.3.1]{green}.
\end{lemma}

Note that (for $w\not\in W(\lambda)$) the definition of
$w\times\gamma$ depends on the choice of $\caF$.

\begin{proof}
Uniqueness is immediate from Proposition
\ref{p:gammarigid}.
For existence, let 
\begin{equation}
\label{e:tau}
\tau=(w\cdot\lambda-\lambda)+(\rho_i(w\lambda)-\rho_i(\lambda))-
(2\rho_{i,c}(w\lambda)-2\rho_{i,c}(\lambda))
\end{equation}
(see Definition \ref{d:reg}).
By Definition \ref{d:Waction}, $\tau$ is in the root lattice, so
we may view it as a character of $H$ which is trivial on  $Z(\tG)$.
An easy calculation shows that $(H,\Gamma\otimes\tau,w\cdot\lambda)$
is a genuine regular character with the same central character as
$\gamma$.   This gives existence.

If $w\not\in W(\lambda)$ then the infinitesimal characters of $\gamma$
and $w\cdot\gamma$ are different by the final assertion of Lemma \ref{l:crossF}.
If $w\in W(\lambda)$ then $w\times\gamma$ as defined in \cite[Definition 8.3.1]{green} 
is of the form $(H,*,w\lambda)$. 
By Lemma \ref{l:crossF}  $w\cdot\lambda=w\lambda$, so our definition of
$w\times\gamma$ is also of this form. The final assertion of the
lemma follows from
Proposition \ref{p:gammarigid} (or a direct comparison of the two
definitions). 
\end{proof}

\subsection{Cayley Transforms}
\label{s:cayley}

We continue to  assume $G$ is simply laced.

Suppose $H$ is a $\theta$-stable Cartan subgroup of $G$ and $\alpha$ is a real
root. Define the Cayley transform $H_\alpha=c_\alpha(H)$
as in \cite[\S 11.15]{overview}. 
%(The choices are such that $H_\alpha$ is
%well-defined. The construction of \cite[Definition 8.3.4]{green},
%which only determines $H_\alpha$ up to $K$-conjugacy, would suffice.)
In particular $\ker(\alpha)=\h_\alpha\cap \h$ is of
codimension one in $\h$ and $\h_\alpha$.
We say a  root $\beta$ of $H_\alpha$ is a Cayley transform of $\alpha$
if $\ker(\beta)=\h\cap\h_\alpha$; there are two such roots
$\pm\beta$, which are necessarily noncompact imaginary.

For $\lambda\in\h^*$ and $\alpha,\beta$ as above define
$c_{\alpha,\beta}(\lambda)\in\h^*_{\alpha}$:
\begin{subequations}
\renewcommand{\theequation}{\theparentequation)(\alph{equation}}  
\begin{align}
\label{e:ctdifferential}
c_{\alpha,\beta}(\lambda)|_{\h\cap\h_\alpha}
&=\lambda|_{\h\cap\h_\alpha}\\
\langle c_{\alpha,\beta}(\lambda),\ch\beta\rangle&=
\langle\lambda,\ch\alpha\rangle.
\end{align}
If $\alpha$ is a noncompact imaginary  root  define $H^\alpha$ and
$c^{\alpha,\beta}$ similarly.
It is clear that
$c^{\alpha,-\beta}(\lambda)=s_\beta c^{\alpha,\beta}(\lambda)$,
$c_{\alpha,-\beta}(\lambda)=s_\beta c_{\alpha,\beta}(\lambda)$,
and
$c^{\beta,\alpha}c_{\alpha,\beta}(\lambda)=
c_{\beta,\alpha}c^{\alpha,\beta}(\lambda)=\lambda$. 
\end{subequations}

If $\tH$ is the preimage of $H$ in $\tG$ 
then $c^\alpha(\tH)$ and
$c_\alpha(\tH)$ are defined to be the inverse images of the
corresponding Cartan subgroups of $G$.

\begin{definitionplain}
\label{d:precayley}
Let $\gamma=(\tH,\Gamma,\lambda)$ be a genuine regular character.
Suppose  $\alpha$ is a
noncompact imaginary root. Let $\tH^\alpha=c^\alpha(\tH)$ and 
suppose $\beta$ is a Cayley
transform of $\alpha$. 
Then the  Cayley transform $c^{\alpha,\beta}(\gamma)$ of $\gamma$ with respect to $\alpha$ and $\beta$
is any regular character of the form
$(\tH^\alpha,\Gamma',c^{\alpha,\beta}(\lambda))$
with the same central character as $\gamma$. 
(Existence and uniqueness are addressed in Proposition 
\ref{p:cayley}.)
If $\alpha$ is a real root 
$c_{\alpha,\beta}(\gamma)$ is defined similarly.
\end{definitionplain}

Define real roots of type I
and II as in \cite[Definition 8.3.8]{ic4}.

\begin{proposition}
\label{p:cayley}
In the setting of Definition \ref{d:precayley},
suppose $\alpha$ is a noncompact
imaginary root and $\beta$ is a Cayley transform of $\alpha$.
Then $c^{\alpha,\beta}(\gamma)$ exists and is unique.
In this case $s_\beta\in W(G,H^\alpha)$ and
$c^{\alpha,-\beta}(\gamma)=s_\beta c^{\alpha,\beta}(\gamma)$. 

Suppose  $\alpha$ is real and $\beta$ is a Cayley transform of $\alpha$.
If $\langle\lambda,\ch\alpha\rangle\not\in\Z+\frac12$ then
$c_{\alpha,\beta}(\gamma)$ does not exist.

Assume $\langle\lambda,\ch\alpha\rangle\in\Z+\frac12$.
If $\alpha$ is  type II  then
$c_{\alpha,\beta}(\gamma)$ exists and is unique,
$s_\beta\in W(G,H_\alpha)$, and 
$c_{\alpha,-\beta}(\gamma)=s_\beta c_{\alpha,\beta}(\gamma)$. 
If  $\alpha$ is  type I, then precisely one of
$c_{\alpha,\pm\beta}(\gamma)$ exists, and is unique.
\end{proposition}

\begin{proof}
If the indicated Cayley transforms exist, they are unique by
Proposition \ref{p:gammarigid}.  Existence is more subtle, however.
As in Section \ref{s:regularCharacters} let $\Delta^i$ be the set of
all (not necessarily integral) roots.  We say an imaginary root
$\alpha$ is {\it imaginary-simple} (with respect to $\lambda$) if it
is simple for
$\{\alpha\in\Delta^i\,|\,\langle\lambda,\ch\alpha\rangle>0\}$.

First assume $\alpha$ is noncompact imaginary and imaginary-simple.
Then the existence of $c^{\alpha,\beta}(\gamma)$ is given
by \cite[Section 4]{ic1}.  We can be more explicit
in our setting as follows.  Write
$c^{\alpha,\beta}(\gamma)=(\tH^\alpha,\Gamma',c^{\alpha,\beta}(\lambda))$.
By Lemma \ref{l:repT}, it is enough to describe
the restriction of $\Gamma'$ to $Z(\tH)$.  Set
\begin{subequations}
\renewcommand{\theequation}{\theparentequation)(\alph{equation}}  
\begin{align}
\Gamma'(g)&=\Gamma(g)\quad (g\in Z(\tH)\cap Z(\tH^\alpha))\\
\langle d\Gamma',\ch\beta\rangle
&=
\langle \lambda,\ch\alpha\rangle.
\end{align}
\end{subequations}
The first condition amounts to specifying $\Gamma'$ on $Z(\tG)$ and
$d\Gamma'|_{\h\cap\h^\alpha}$. 
It is straightforward to see that such $\Gamma'$ exists; see 
\cite[Lemma 4.32]{aherb}. To show that
$(\tH^\alpha,\Gamma',c^{\alpha,\beta}(\lambda))$ is a regular
character we invoke that $\alpha$ is imaginary-simple and apply
\cite[Proposition 5.3.4]{green}.

Now suppose $\alpha$ not necessarily imaginary-simple.
Choose
 $w\in W(\g,\h)$ making $\alpha$ imaginary-simple.
Then $w\inv\times\gamma$
is defined (Section \ref{s:cayley}) and is of the
form $(\tH,*,w\inv\cdot\lambda)$, where $*$ represents some
representation we need not
specify. By definition $w\inv \cdot\lambda$
is in the same Weyl chamber as $w\inv\lambda$ so $\alpha$ is imaginary-simple
with respect to $w\lambda$.
Then, letting $c=c^{\alpha,\beta}$,
$c(w\inv\times\gamma)$ is defined, and is of the form
$(\tH^\alpha,*,c(w\inv\cdot\lambda))$.
It is clear that $\wt wc(w\inv\cdot\lambda)=c(\lambda)$ for
some $\wt w\in W(\g,\h^\alpha)$. 
Let $\caF^\alpha=c\caF$ use it to define the cross action for $H^\alpha$.
It is clear that
$c(w\inv\cdot\lambda)$ and $c(\lambda)$ differ by a sum
of roots, and $\wt w\cdot c(w\inv\cdot\lambda)=c(\lambda)$. 
Consider
\begin{equation}
\label{e:calphabeta}
\wt w\times c(w\inv\times\gamma).
\end{equation}
This is a regular character of the form
$(\tH,*,c(\lambda))$, and has the same central character
as $\gamma$.  This establishes existence in this case.

Now suppose $\alpha$ is real.
As in the preceding case, using the cross action we may assume
$\beta$ is imaginary-simple. 
in which case $\Gamma'$ is constructed in
\cite[Section 4]{ic1}, although less explicitly. In this case
$\Gamma'$ must satisfy
\begin{subequations}
\label{e:realcayley}
\renewcommand{\theequation}{\theparentequation)(\alph{equation}}  
\begin{align}
\Gamma'(g)&=\Gamma(g)\quad (g\in Z(\tH)\cap Z(\tH_\alpha))\\
\langle d\Gamma',\ch\beta\rangle
&=
\langle \lambda,\ch\alpha\rangle+\langle\rho_i(c(\lambda))-2\rho_{i,c}(c(\lambda)),\ch\beta\rangle
\end{align}
with notation as in \eqref{e:regchar2}, 
and with $c(\lambda)=c_{\alpha,\beta}(\lambda)$.
See \cite[Definition 8.3.14]{green}.
Let  $\eta=\langle\lambda,\ch\alpha\rangle$ and 
$d=\langle\rho_i(c(\lambda))-2\rho_{i,c}(c(\lambda),\ch\beta\rangle\in\Z$.

Let $m=\wt\exp(\pi i\ch\beta)$ where $\wt\exp:\h_0\rightarrow \tH$ is
the exponential map. (This is the element $m_\alpha$ of
\eqref{e:metalpha}). 
By Lemma \ref{l:wtexp} $m^2=-1$, so $\Gamma(m)=\pm i$ since $\Gamma$
is genuine. 
On the other hand \eqref{e:realcayley}(b) gives
\begin{equation}
\Gamma'(m)=\exp(\pi i(\eta+d)).
\end{equation}
\end{subequations}
Since  $d\in \Z$ this forces
$\eta\in \Z+\frac12$, so assume this holds.

In our setting there is no further condition if $\wt{m}$ is not
contained in $Z(\tH)$, i.e. if $\alpha$ is type II. 
If $\alpha$ is of type I then $\Gamma'(m)$ is given by both 
(a) and (b).
Changing $\beta$ to $-\beta$ does not change  the right hand side of
(c), but replaces $m$ with $m\inv$ and therefore
$\Gamma'(m)$ with $-\Gamma'(m)$.
Therefore
(c) holds for precisely one choice of
$\beta$ or $-\beta$.
\end{proof}

% Thus if $\alpha$ is imaginary, we could define $c^\alpha(\gamma)$ to
% be a pair of regular characters, related by $W(G,H^\alpha)$. If
% $\alpha$ is real we define $c_\alpha(\gamma)$ to be either a pair of
% regular characters related by $s_\beta\in W(G,H_\alpha)$ (if $\alpha$
% is type II) or the unique character $c_{\alpha,\beta}(\gamma)$ given
% by the Proposition. In each case the Cayley  transform gives a single well
% defined $K$-conjugacy class of regular characters.
% Write $\overline\gamma$ for the $K$-conjugacy class of $\gamma$.

The proposition shows that, unlike the linear case, $c^{\alpha,\beta}(\gamma)$
is never multivalued.  Like the linear case, the choice of $\beta$
can only affect the Cayley transform up to $K$-conjugacy.

\begin{definitionplain}
\label{d:precayley2}
Fix a genuine regular character
 $\gamma=(\tH,\Gamma,\lambda)$, and 
suppose $\alpha\in \Delta_{\frac12}^{i,-}(\gamma)$.
Choose a Cayley transform  $\beta$ of $\alpha$ and let
$c^\alpha(\gamma)$ denote $c^{\alpha,\beta}(\gamma)$.  
Define
$c^\alpha(cl(\gamma))$ to be $cl(c^{\alpha,\beta}(\gamma))$.
Although there is a choice of $\beta$ in the definition of
$c^\alpha(\gamma)$, 
Proposition \ref{p:cayley} shows that $c^\alpha(cl(\gamma))$
independent of this choice.

Dually, suppose $\alpha\in\Delta_{\frac12}^{r,-}(\gamma)$.
Choose a Cayley transform 
$\beta$ of $\alpha$, let
$c_\alpha(\gamma)$ denote $c_{\alpha,\beta}(\gamma)$, and
define
$c_\alpha(cl(\gamma))$ to be $cl(c_{\alpha,\beta}(\gamma))$.
\end{definitionplain}

\medskip

We now introduce the {\it abstract} Cartan subalgebra and Weyl group
as in \cite[2.6]{ic4}.  Thus we fix once and for all a Cartan subalgebra
$\h_a$, a set of positive roots $\Delta^+_a$ of $\Delta_a=\Delta(\g,\h_a)$, and
let $W_a=W(\g,\h_a)$. Suppose $\h$ is any Cartan subalgebra of $\g$
and $\lambda$ is a regular element of $\h^*$. 
Let $\phi_\lambda:\h_a^*\rightarrow \h^*$ be
the unique  isomorphism which is inner for the complex group, such
that $\phi_\lambda\inv(\lambda)$ is $\Delta^+_a$-dominant.
This induces isomorphisms $\Delta_a\simeq\Delta$ and
$W_a\simeq W$, written $\alpha\rightarrow \alpha_\lambda$ and
$w\rightarrow w_\lambda$ respectively.

Fix a regular infinitesimal character and let $\lambda_a$ be the
corresponding $\Delta^+_a$-dominant element of $\h^*_a$ via the
Harish-Chandra homomorphism. Choose a set $\caF_a\subset \h_a^*$ as in
Section \ref{s:F}.

The cross action of $W_a$ is defined using $\caF_a$ as follows. Suppose
$\gamma=(\tH,\Gamma,\lambda)$ is a genuine regular character and
$\phi_\lambda\inv(\lambda)\in\caF_a$. Then for $w\in  W_a$ define
\begin{equation}
\label{e:abstractca}
w\times\gamma=w_\lambda\inv\times\gamma.
\end{equation}
(The cross action on the right hand side is defined using
$\phi_\lambda(\caF_a)$). 
Compare \cite[Definition 4.2]{ic4}.  It is easy to see this is an
action on regular characters.
If $\gamma\sim\gamma'$ then it is easy to see that $w\times\gamma\sim
w\times\gamma'$, so this induces an 
action on 
$K$-orbits of regular characters.

Now fix $\gamma=(\tH,\Gamma,\lambda)$ and suppose $\alpha\in
\Delta(\g,\h_a)$. We say $\alpha$ is real, imaginary, etc.  for $\gamma$ if
this holds for $\alpha_\lambda$.
Write $\Delta^{i,\pm}(\gamma)_a$
and $\Delta^{i,\pm}(\gamma)_a$
for the subsets of $\Delta_a$, pulled back from the corresponding
subsets of $\Delta$.  Suppose $\alpha\in \Delta^{i,-}_{\frac12}(\gamma)_a$. 
Define $c^\alpha(\gamma)=c^{\alpha_\lambda}(\gamma)$;
recall there are   two choices $c^{\alpha,\pm\beta}$ of
$c^\alpha(\gamma)$, and $c^\alpha(cl(\gamma))$ is well defined.
For $\alpha\in \Delta^{r,-}_{\frac12}(\gamma)_a$ define $c_\alpha$ 
similarly. 
Then
\begin{subequations}
\label{e:inverses}
\begin{align} 
c^\alpha(c_\alpha(\gamma))&\sim\gamma\quad(\alpha\in \Delta^{i,-}_{\frac12}(\gamma)_a) \\
c_\alpha(c^\alpha(\gamma))&\sim\gamma\quad(\alpha\in \Delta^{r,-}_{\frac12}(\gamma)_a).
\end{align}
\end{subequations}
To be precise these equalities hold for both choices of
$c^{\alpha,\pm\beta}$ of $c^\alpha$, and $c_{\alpha,\pm\beta}$ of $c_\alpha$.
Similar
equations will arise below (for example, \eqref{e:ca}) and they are to be understood
in the same way.

% We iterate Cayley transforms as in \cite[Section 7]{ic4}. 

\begin{definitionplain}
A real admissible subspace for $\gamma$ is a   
a subspace $\frs\subset \h_a^*$ satisfying the following conditions.

\noindent (1) $\frs$ is spanned by by a set of real roots
$\alpha_1,\dots, \alpha_n$,

\noindent (2) the iterated Cayley transform
$c_{\alpha_1}\circ\dots\circ c_{\alpha_n}(\gamma)$ is defined for
some (equivalently, any) choices.
\end{definitionplain}

\begin{comment}
Condition (2) amounts to the requirement that for $i=1\dots, n-1$,
\begin{equation}
\alpha_i\in
\Delta^{r,-}_{\frac12}(c_{\alpha_{i+1}}\circ\dots\circ
c_{\alpha_n}(\gamma))_a.
\end{equation}
This condition is independent of the intermediate choices involved in
defining $c_{\alpha_{i+1}}\circ\dots\circ c_{\alpha_n}(\gamma)$.
\end{comment}

\begin{lemma}
\label{l:cayleybasic}
A subspace $\frs$ of $\h_a^*$ is a real-admissible subspace for
$\gamma$
if and only if it has a 
basis consisting  of orthogonal roots of
$\Delta^{r,-}_{\frac12}(\gamma)_a$.
This basis is unique (up to ordering and signs).
\end{lemma}

\begin{proof}
The discussion after \cite[Definition 5.3]{ic4} shows that any
real admissible sequence must be orthogonal.
It is clear that the span of a set of strongly orthogonal elements of
$\Delta^{r,-}_{\frac12}$ is admissible. In the simply laced case
orthogonal is equivalent to strongly orthogonal, and the first part
follows easily. It is easy to see the only roots in the span of a set
of strongly orthogonal roots  $\{\alpha_1,\dots,\alpha_n\}$ 
are $\pm\alpha_i$, and the second
assertion follows from this.
\end{proof}

\begin{definitionplain}
\label{d:cayleyr}
Suppose $\frs\subset\h_a^*$ is a real-admissible subspace for
$\gamma$. Let
$c_{\frs}(\gamma)$ denote $c_{\alpha_1}\circ\dots\circ c_{\alpha_n}(\gamma)$
where $\{\alpha_1,\dots, \alpha_n\}\subset \Delta^{r,-}_{\frac12}(\gamma)_a$ is
a strongly orthogonal basis of $\frs$. 
Define $c_\frs(cl(\gamma))=cl(c_\frs(\gamma))$.  By
Proposition \ref{p:cayley} and Lemma \ref{l:cayleybasic}, this is well-defined
independent of all choices.
\end{definitionplain}

Imaginary admissible subspaces and iterated imaginary Cayley
transforms are defined similarly:

\begin{definitionplain}
\label{d:cayleyi}
An imaginary admissible subspace of $\h_a^*$ is one which has a basis of
strongly 
orthogonal roots of $\Delta^{i,-}_{\frac12}(\gamma)_a$. If $\frs$ is such a subspace
let
$c^{\frs}(\gamma)$ denote $c^{\alpha_1}\circ\dots\circ c^{\alpha_n}(\gamma)$
where $\{\alpha_1,\dots, \alpha_n\}\subset \Delta^{r,-}_{\frac12}(\gamma)_a$ is
an orthogonal basis of $\frs$.
Define $c^\frs(cl(\gamma))=cl(c^\frs(\gamma))$.
\end{definitionplain}

The analogue of \cite[Lemma 7.11]{ic4} is straightforward.

\begin{lemma}
\label{l:ca}
Let $\gamma = (\tH, \Gamma, \lam)$ be a genuine regular character
for $\wt G$, with $\phi_\lambda\inv(\lambda)\in\caF_a$.
 Fix $w\in W_a$ and $\frs\subset \frh_a^*$.  Then
$\frs$ is an imaginary admissible subspace for $\gamma$ if and only
if $w\frs$ is an imaginary admissible subspace for $w \times
\gamma$. If this holds then
\begin{subequations}
\label{e:ca}
\renewcommand{\theequation}{\theparentequation)(\alph{equation}}  
\begin{equation}
c^{w\frs}(w\times\gamma) \sim w\times c^{\frs}(\gamma).
\end{equation}
Similarly 
$\frs$ is a real admissible subspace for $\gamma$ if and only
if $w\frs$ is a real admissible subspace for $w \times
\gamma$, in which case
\begin{equation}
c_{w\frs}(w\times\gamma) \sim w\times c_{\frs}(\gamma).
\end{equation}
\end{subequations}
\end{lemma}

\begin{proof}
The facts about admissible subspaces are elementary, and similar to
the linear case. Thus the left hand side of \eqref{e:ca}(a) is defined, 
and by Proposition \ref{p:gammarigid} the two sides of 
\eqref{e:ca}(a) are actually equal after making
the obvious choices. Equivalently
the two sides are $K$-conjugate for any choices. The same holds for 
\eqref{e:ca}(b).
\end{proof}

Define the cross stabilizer of $\gamma$ to be
\begin{equation}
\label{e:crossstab1}
\begin{aligned}
W(\gamma)&=\{w\in W(\lambda)\,|\, w\times\gamma\sim\gamma\}\\
&=\{w\in W(\lambda)\cap W(G,H)\,|\,
w\times\gamma=w\gamma\}.
\end{aligned}
\end{equation}
See \cite[Definition 4.13]{ic4}.
The next result is quite different from what one encounters
in the linear case \cite[Proposition 4.14]{ic4}.

\begin{lemma}
\label{l:crossstab}
\begin{equation}
W(\gamma)=W(\lambda)\cap W(G,H).
\end{equation}
\end{lemma}

\begin{proof}
If $w\in W(\lambda)$ then
$w\times\gamma$ is of the form $(H,*,w\lambda)$ by Lemma
\ref{l:crossF}.
If $w\in W(G,H)$ the same holds for $w\gamma$. 
Since $w\times\gamma$ and $w\gamma$ have the same central character
Proposition \ref{p:gammarigid} implies $w\times\gamma=w\gamma$.
\end{proof}

We need a more explicit version of the lemma, for which we use some
notation from \cite[Section 3]{ic4}.
Recall (Section \ref{s:bigradings}) $\Delta(\lambda)$ is 
the set of integral roots.
Let
$\Delta^+=\{\alpha\in\Delta(\lambda)\,|\,\langle\lambda,\ch\alpha\rangle>0\}$.
Let $\Delta^r,\Delta^i$ be the real and imaginary roots as usual
(Section \ref{s:regularCharacters}), and
let $\Delta^i(\lambda)=\Delta^i\cap\Delta(\lambda)$,
let $\Delta^r(\lambda)=\Delta^r\cap\Delta(\lambda)$ as in Section
\ref{s:bigradings}. Let
\begin{subequations}
\renewcommand{\theequation}{\theparentequation)(\alph{equation}}  
\begin{equation}
\rho_i=\frac12\sum_{\alpha\in\Delta^+\cap \Delta^i(\lambda)}\alpha,\quad
\ch\rho_r=\frac12\sum_{\alpha\in\Delta^+\cap \Delta^r(\lambda)}\ch\alpha
\end{equation}
and
\begin{equation}
\Delta^C(\lambda)=\{\alpha\in\Delta(\lambda)\,|\,
\langle\alpha,\ch\rho_r\rangle=
\langle \rho_i,\ch\alpha\rangle=0\}.
\end{equation}
Let 
\begin{equation}
 W^i(\lambda)=W(\Delta^i(\lambda)),
 W^r(\lambda)=W(\Delta^r(\lambda)),
W^C(\lambda)=W(\Delta^C(\lambda)).
\end{equation}
\end{subequations}
Finally let $W^C(\lambda)^\theta$ be the fixed points of $\theta$
acting on $W^C(\lambda)$.

\begin{proposition}
\label{p:crossstab}
\begin{equation}
W(\gamma)=W^C(\lambda)^\theta\ltimes (W^i(\lambda)\times W^r(\lambda)).
\end{equation}
\end{proposition}

\begin{proof}
By Lemma \ref{l:crossstab} $W(\gamma)=W(\lambda)\cap W(G,H)$, and as
in \cite[Proposition 4.14]{ic4} this equals
$W(\lambda)^\theta\cap W(G,H)$.  Propositions 3.12 and 4.16 of
\cite{ic4}  compute the two terms on the right hand side; the result
is
\begin{equation}
W(\gamma)=W^C(\lambda)^\theta\ltimes ([W^i(\lambda)\cap W(M,H)]\times W^r(\lambda)).
\end{equation}
As in \cite[Proposition 4.16(d)]{ic4}
\begin{equation}
W(\Delta^{i,+}(\gamma))\subset W(M,H)\subset\Norm_{W^i}(\Delta^{i,+}(\gamma)).
\end{equation}
Intersecting with $W^i(\lambda)$ gives
\begin{equation}
W(\Delta^{i,+}(\gamma))\subset W(M,H)\cap W^i(\lambda)\subset\Norm_{W^i(\lambda)}(\Delta^{i,+}(\gamma)).
\end{equation}
By 
\eqref{e:gradings}(b),
$\Delta^{i,+}(\gamma)=\Delta^i(\lambda)$, so this gives
\begin{equation}
W^i(\lambda)\subset W(M,H)\cap W^i(\lambda)\subset 
\Norm_{W^i(\lambda)}(\Delta^i(\lambda)).
\end{equation}
The outer terms are the same, so $W(M,H)\cap W^i(\lambda)=W^i(\lambda)$.
This completes the proof.  
\end{proof}

In the setting of the abstract Weyl group the abstract cross
stabilizer of $\gamma$ is
\begin{equation}
\label{e:crossstab-a}
W(\gamma)_a=\{w\in W(\lambda_a)\cap W(G,H)_a\,|\, w\times\gamma=w\gamma\}.
\end{equation}
where $W(G,H)_a=\{w\in W_a\,|\, w_\lambda\in W(G,H)\}$.
The obvious analogues of Lemma \ref{l:crossstab} and Proposition \ref{p:crossstab} hold.

\subsection{Blocks}
\label{s:blocks}
We now turn to blocks of regular characters (Definition 
\ref{d:blockregchar}) still assuming $G$ is simply laced.

\begin{lemma}
\label{l:blocksclosed}
Blocks of genuine regular characters
are  closed under $K$-conjugacy, the cross action of $W(\lambda_a)$,
and %(any choices of) 
the Cayley transforms of Definitions
\ref{d:cayleyr} and \ref{d:cayleyi}. 
\end{lemma}

\begin{proof}  Closure under $K$-conjugacy follows
from the definition (and Theorem \ref{t:lang}), and
the case of the cross action of $W(\lambda_a)$ and real Cayley
transforms are covered by \cite{green}. 
The case of noncompact imaginary Cayley transforms then follows from
\eqref{e:inverses}.
\end{proof}

\begin{proposition}
\label{p:blocks}
Blocks of genuine regular characters with infinitesimal character
$\lam_a$ are the smallest sets closed 
under $K$-conjugacy, the cross action of $W(\lambda_a)$,
and  the Cayley transforms of Definitions
\ref{d:cayleyr} and \ref{d:cayleyi}.   
\end{proposition}

\begin{proof}
We argue as in \cite[9.2.11]{green}.  More precisely, using
Lemma \ref{l:blocksclosed} and \cite[9.2.10]{green} (whose
proof carries over in our setting), we are
reduced to showing that if $\irr(\gamma')$ is a composition factor of 
$\std(\gamma)$, then $\gamma$ and $\gamma'$ may be related by the
relations listed in the proposition.  We proceed by induction on
a modification of
the usual integral length of $\gamma = (\tH, \Gamma, \lambda)$:
\[
l(\gamma) = \frac12\left |\{ \alpha \in \Delta^+_{\frac12}(\lam) \; | \; 
\theta(\alpha) \notin \Delta^+_{\frac12}(\lam) \} \right | + \dim(\fra_\R).
\]
(Recall we are in the simply laced case so 
$\Delta^+_{\frac12}(\lambda) = \Delta^+(2\lam)$.) 

Suppose there is a root $\alpha$ which is simple in
$\Delta^+_{\frac12}(\lambda)$, and either
$\alpha\in\Delta^{r,-}_{\frac12}(\lambda)$ or
$\alpha$ is complex, integral, and $\theta(\alpha)<0$.
In the first case the argument of \cite[9.2.11]{green} holds in this
case. In the second, since $\alpha$ is also simple for
$\Delta^+(\lambda)$,  the argument of \cite[9.2.11]{green} applies
directly.

So suppose $\std(\gamma)$ is reducible, but
no roots of the kind just described exist.  Then (see 
\cite[Remark 7.7]{rt3}) there exists a (nonintegral)
complex root $\alpha$ which is simple in $\Delta^+_{\frac12}(\lam)$
such that $\theta(\alpha) \notin \Delta^+_{\frac12}(\lam)$.
Choose a family $\caF$  and
define the abstract cross action in $W_a$ as in \eqref{e:abstractca}.
Consider
$\eta = s_{\alpha_\lam} \times \gamma$ (where now $s_{\alpha_\lam}$ is 
the reflection in the abstract Weyl group through the abstract
root $\alpha_\lam$ corresponding to $\alpha$;
see the discussion about Definition \ref{d:precayley}).
Then $l(\eta) < l(\gamma)$, but
$\eta$ has  infinitesimal character $s_{\alpha_\lam}\cdot \lam_a \neq \lam_a$.  
The calculations of \cite[Section 4]{ic1} show that
$\std(\eta) = \psi(\std(\gamma))$ for a nonintegral wall crossing
functor $\psi$ which is an equivalence of categories.  
Since $\psi$ is an equivalence, we may write $\psi(\irr(\gamma'))=\irr(\eta')$
using Theorem \ref{t:lang}.
Thus $\irr(\gamma')$ is a composition factor of $\std(\gamma)$
if and only if $\irr(\eta')$ is a composition factor of $\std(\eta)$.
By induction, $\eta$ and $\eta'$ are related by a sequence of the kind
described in the statement of the proposition with 
$W(s_{\alpha_\lam}\cdot \lam_a)$
in place of $W(\lam_a)$.
From Lemma \ref{l:ca}, we conclude that $\gamma$ and $\gamma'$ 
are related by conjugation, Cayley transforms, and the cross
action in $W_a$.
But the last sentence of 
Lemma \ref{l:cross} implies that the cross action must in 
fact be in $W(\lam_a)$, as claimed.
\end{proof}

We now introduce unions of blocks as in \cite{rt3}. Fix
$\lambda_a\in\h_a$ and a  block $\scB$
off genuine regular characters of
 infinitesimal character $\lambda_a$.
Choose a family $\caF\subset\h_a^*$
as above containing $\lam_a$ and use it to define the cross action as 
in \eqref{e:abstractca}.
Recall the definition  (Section \ref{s:bigradings})
of $W_{\frac12}(\lambda_a)$; since $G$ is simply laced this equals
$W(2\lambda_a)$.  
Set
\begin{equation}
\label{e:B}
\bfB = \{w\times\gamma \; | \; w \in W_{\frac12}(\lam_a), \; \gamma \in \scB\}.
\end{equation}
Then $\bfB$ is a union of the blocks discussed above, each of whose
infinitesimal character is in $\caF$.

\begin{corollary}
\label{c:bfB}
The set $\bfB$ of $\eqref{e:B}$ is the
smallest set containing $\caB$ which is
closed under $K$-conjugacy,
the cross action of
$W_{\frac12}(\lambda_a)$, and the
Cayley transforms of Definitions 
\ref{d:cayleyr} and \ref{d:cayleyi}.
\end{corollary}

\sec{Duality}
\label{s:kl}

Throughout this section let $\tG$ be an admissible double cover of a
simply laced real reductive linear group $G$ (Section \ref{s:notation}
and Definition \ref{d:admissible}).  The representation theory of
$\tG$ is sufficiently rigid to reduce the existence of a duality
theory to the duality of  bigradings
(Definition \ref{d:big}).
Theorem \ref{t:gendual} makes this precise.

\begin{definitionplain}
\label{d:weaklymin}
A genuine regular character $\gamma$ (Definition
\ref{d:reg})  is called {\it weakly minimal} if
$\Delta_{\frac12}^{r,-}(\gamma)$ (cf.~\eqref{e:psir+}) is
empty, i.e.~there are no real, half-integral roots satisfying the
parity condition.
Similarly we say $\gamma$ is  weakly
maximal if $\Delta^{i,-}_{\frac12}(\gamma)$ is empty.
\end{definitionplain}

It is easy to see that every block $\scB$ of genuine regular characters
contains a weakly minimal
element: if $\gamma \in \scB$ and $\frs$ is a maximal
real-admissible subspace for $\gamma$
then $c_\frs(\gamma)$ is weakly minimal.  In the same way
weakly maximal elements exist.  The following  uniqueness statement for
these elements  will be important for the proof of Proposition
\ref{p:standardblocks}.  The analogous result in the linear case is
\cite[Lemma 8.10]{ic4}.

Fix  an abstract Cartan subalgebra $\h_a$ and  a regular element $\lambda_a\in\h_a^*$.

\begin{proposition}
\label{p:weaklyminimal}
Let $\scB$ be a block of genuine regular characters for $\tG$
with infinitesimal character $\lambda_a$.
Suppose $\gamma$ and $\gamma'$ are weakly minimal elements 
in $\scB$.  Then there is an element $w \in W(\lam_a)$
such that $\gamma' \sim w \times \gamma$.  An analogous
statement holds for weakly maximal elements.
\end{proposition}

\begin{lemma}
\label{lem:weaklyminimal}
Suppose $\gamma$ is weakly minimal,
$\frs\subset\h_a^*$ is an imaginary admissible subspace for $\gamma$, and
$\fru\subset\h^*_a$ is a real admissible subspace for $c^\frs(\gamma)$.
Then there is an imaginary admissible subspace $\frw$ for
$\gamma$ such that
\begin{equation}
c_{\fru}(c^{\frs}(\gamma))\sim c^{\frw}(\gamma).
\end{equation}
\end{lemma}

\begin{proof}
Suppose $\frs$ is the span a set $S$ of orthogonal roots of 
$\Delta^{i,-}_{\frac12}(\gamma)_a$.
Since $\Delta^{r,-}_{\frac12}(\gamma)_a$ is empty, it follows easily
that
$\Delta^{i,-}_{\frac12}(c^{\frs}(\gamma)_a)=S \cup (-S)$.
Therefore $\fru$ is spanned by a subset $T$ of $S$; take $\frw$ to be
the span of the complement of $T$ in $S$.
\end{proof}

\medskip

\noindent{\bf Proof of Proposition \ref{p:weaklyminimal}.}
As in Proposition \ref{p:blocks} we can write
$w\times\gamma\sim c\gamma'$
where $w\in W(\lambda_a)$ and $c$ is a sequence of Cayley transforms. 
Writing $c=c^{\frs_1}\circ c_{\frs_2}\circ\dots\circ c^{\frs_n}$, and
using the lemma repeatedly, we may assume $n=1$ and $c=c^{\frs}$ for
some imaginary admissible subspace $\frs$ for $\gamma$.
Thus $w\times\gamma\sim c^{\frs}(\gamma')$, or 
by \eqref{e:inverses}, $c_{\frs}(w\times\gamma)\sim\gamma'$.
By Lemma \ref{l:ca} $w\inv\frs$ is a real-admissible subspace for $\gamma$,
contradicting the assumption that $\gamma$ is weakly minimal unless
$\frs=0$ and $w\times\gamma\sim\gamma'$, as claimed.
\qed

Recall  $\ol{\scB}$ denotes the set of $K$-orbits in $\scB$,
and fix a weakly minimal element $\gamma$ of $\scB$.
Let $S_i(\gamma)_a$ be the set of subspaces of $\h_a^*$ which are
imaginary admissible subspaces 
for $\gamma$.
If $u\in W(\gamma)_a$ 
(cf.~\ref{e:crossstab-a})
then $u\times\gamma\sim\gamma$, so by 
Lemma \ref{l:ca} if $\frs$ is contained in $S_i(\gamma)_a$ then so is $u\frs$. 
Therefore $W(\gamma)_a$ acts on 
on $S_i(\gamma)_a\times W(\lambda_a)$:
\begin{equation}
\label{e:uaction}
u\cdot(\frs,w)=(u\frs,wu\inv)\quad (u\in W(\gamma)_a).
\end{equation}
The map 
\begin{equation}
\psi_i(\frs,w)=cl(c^{w\frs}(w\times\gamma))\in
\ol\scB
\end{equation}
is easily seen to be well-defined, and factors to $(S_i(\gamma)_a\times
W(\lambda_a))/W(\gamma)_a$. This gives a well-defined map
\begin{subequations}
\label{e:psis}  
\renewcommand{\theequation}{\theparentequation)(\alph{equation}}  
\begin{equation}
\label{e:psii}
\psi_i: (S_i(\gamma)_a\times W(\lambda_a))/W(\gamma)_a\rightarrow \ol\scB.
\end{equation}
If $\gamma$ is a weakly maximal element of $\scB$, define
$S_r(\gamma)_a$ similarly (as real admissible subspaces), and the
analogous map
\begin{equation}
\label{e:psir}
\psi_r: (S_r(\gamma)_a\times W(\lambda_a))/W(\gamma)_a\rightarrow \ol\scB.
\end{equation}
\end{subequations}

\begin{proposition}
\label{p:standardblocks}
If $\gamma$ is a weakly 
minimal (respectively weakly maximal) element of $\scB$, the map of
\eqref{e:psii}
(respectively (b))
is a bijection.
\end{proposition}

\begin{proof}
We only consider the first case, the second is similar. Fix a
weakly minimal element $\gamma\in\scB$.

We first prove injectivity.
It is clear that the cross action of $W_a$
preserves the properties of being weakly minimal or weakly maximal
(for example by Lemma \ref{l:ca}).
Suppose
$\psi_i(\frs,w)=\psi_i(\frs',w')$, so
$c^{w\frs}(w\times\gamma)\sim c^{w'\frs'}(w'\times\gamma)$.
By repeated applications of \eqref{e:inverses} this is equivalent to
\begin{equation}
c_{w'\frs'}c^{w\frs}(w\times\gamma)\sim w'\times\gamma.
\end{equation}
By Lemma \ref{lem:weaklyminimal} 
this gives
\begin{equation}
c^{\frw}(w\times\gamma)\sim w'\times\gamma
\end{equation}
for some imaginary admissible subspace $\frw$ for $\gamma$. 
Then $\frw$ is a real admissible subspace for $w'\times\gamma$. 
This 
contradicts the assumption that $\gamma$ (and therefore
$w'\times\gamma$)
is weakly minimal unless
$\frw=0$, in which case $w\frs=w'\frs'$ and 
$w\times\gamma\sim w'\times\gamma$.
Let $u=(w')\inv w\in W(\lambda_a)$. Then $u\times\gamma\sim\gamma$, 
i.e. $u\in W(\gamma)_a$ (cf.~\ref{e:crossstab-a}).
Then $(\frs',w')=(u\frs,wu\inv)$.
This proves injectivity.

For surjectivity, fix $\delta\in \scB$.
Let $\frs$ be a maximal real admissible subspace for $\delta$, so 
$c_\frs(\delta)$ is weakly minimal.
By Proposition \ref{p:weaklyminimal},
$c_\frs(\delta)\sim w\times\gamma$ for some $w\in W(\lambda_a)$.
Then $\delta\sim\psi_i(w\inv\frs,w)$.
\end{proof}

We next obtain an analogous result for the set $\bfB$ 
of \eqref{e:B}
by replacing $W(\lambda_a)$
with $W_{\frac12}(\lambda_a)$ in \eqref{e:psis}(a) and (b). 
Let $\ol \bfB$ denote the set of $K$-conjugacy classes in $\bfB$.
For $\gamma$ a weakly minimal element of $\bfB$ 
define a map
\begin{subequations}
\renewcommand{\theequation}{\theparentequation)(\alph{equation}}  
\begin{equation}
\label{e:psii2}
\psi_i: (S_i(\gamma)_a\times W_{\frac12}(\lambda_a))/W(\gamma)_a\rightarrow 
\ol \bfB
\end{equation}
taking a representative $(\frs,w)$ on the left-hand side to
$c^{w\frs}(w\times\gamma)$.
If $\gamma$ is weakly maximal define
\begin{equation}
\label{e:psir2}
\psi_r: (S_r(\gamma)_a\times W_{\frac12}(\lambda_a))/W(\gamma)_a\rightarrow 
\ol \bfB
\end{equation}  
similarly.
\end{subequations}

\begin{proposition}
\label{p:standardblocks2}
If $\gamma$ is a weakly minimal (respectively weakly maximal)
element of $\bfB$, the map of \eqref{e:psii2} (respectively
(b)) is a bijection.
\end{proposition}

\begin{proof}
The proof is essentially the same as that of Proposition
\ref{p:standardblocks}.  
For injectivity (see the proof of Proposition \ref{p:standardblocks}) 
we need that if $u\in W_{\frac12}(\lambda_a)\cap W(G,H)_a$
satisfies $u\times\gamma=u\gamma$ then $u\in W(\gamma)
=\{w\in W(\lambda_a)\cap W(G,H)_a\,|\, w\times\gamma=\gamma\}$
(cf.~\ref{e:crossstab1}).
This is almost obvious, except that $u\in W_{\frac12}(\lambda_a)$, and
not necessarily $W(\lambda_a)$. It suffices to show that if $u\in
W_{\frac12}(\lambda_a)\bs W(\lambda_a)$ then
$u\times\gamma\ne u\gamma$.
Write $\gamma=(\tH,\Gamma,\lambda)$.
The left hand side is of the form
$(\tH,*,u\cdot\lambda)$, the right hand side is of the form
$(\tH,*,u\lambda)$.
The result is now immediate from Lemma \ref{l:crossF}, which says that
$u\cdot\lambda\ne u\lambda$.

Surjectivity also follows as in Proposition \ref{p:standardblocks},
using a version of Proposition \ref{p:weaklyminimal} for $\bfB$.
\end{proof}

Let $\g_d$ be the derived algebra of $\g$, and for $\h$ a Cartan
subalgebra let $\h_d=\h\cap\g_d$.

\begin{theorem}
\label{t:gendual}
Recall $\tG$ is an admissible two-fold cover of a simply laced, real
reductive linear group.
Let $\b$
be a block of genuine representations of $\tG$ with regular
infinitesimal character.  Fix a weakly minimal element
$\gamma=(H,\Gamma,\lambda)$ such that $\irr(\gamma) \in \b$.
Suppose we are given:

\medskip

\noindent (a)
an admissible cover $\wt G'$ of a simply laced,
real reductive linear group $G'$;

\smallskip

\noindent (b)
a genuine regular character $\gamma'  = (H', \Gamma', \lam')$ of
$\wt G'$.

\smallskip

Let 
$$\Delta(\lambda')=\{\alpha\in\Delta(\g',\h')\,|\,\langle\lambda',\ch\alpha\rangle\in\Z\}$$
and
$$\Delta_{\frac12}(\lambda')=\{\alpha\in\Delta(\g',\h')\,|\,\langle2\lambda',\ch\alpha\rangle\in\Z\}.$$ 
Let $\theta$ and $\theta'$ denote the Cartan involutions of $\g$ and $\g'$
respectively.
We also assume we are given

\smallskip
\noindent (c)
an isomorphism $\phi$ from $\h'_d$ to
$\C\langle\Delta_{\frac12}(\lambda)^\vee\rangle\subset\h$
satisfying
\begin{subequations}
\renewcommand{\theequation}{\theparentequation)(\alph{equation}}
\label{e:assumptions}
\begin{equation}
\label{e:phitheta}
\phi(\theta'(X'))=-\theta(\phi(X'))\quad\text{for all }X'\in \h'_d,
\end{equation}
\begin{equation}
\label{e:phiWmatch}
\phi^*(\Delta(\lam))=\Delta(\lambda'),
\end{equation}
\begin{equation}
\label{e:phiWmatch2}
\phi^*(\Delta_{\frac12}(\lam))=\Delta_{\frac12}(\lambda').
\end{equation}
\end{subequations}
Let $\b'=\b(\irr(\gamma'))$.
Then  $\b'$ is dual to $\b$
(Definition \ref{d:vogandual}).
\end{theorem}
\medskip

Note that \eqref{e:assumptions} implies that
the the bigradings $\ghalf(\gamma)$ and
$\ghalf(\gamma')$ of Definition \ref{d:big} are dual.
The key point is that this is enough to  imply $\b'$ is dual to $\b$. 

Before turning to the proof of the theorem, we define the
bijection $\Phi:\b\rightarrow\b'$.
By \eqref{e:assumptions}(a) and (b), we obtain
a bijection
\begin{equation}
\Delta^i(\lambda)\leftrightarrow \Delta^r(\lambda').
\end{equation}
By  Proposition \ref{p:gradings},
$\Delta^{i,+}_{\frac12}(\gamma)=\Delta^i(\lambda)$ 
and
$\Delta^{r,+}_{\frac12}(\gamma')=\Delta^r(\lambda')$, so there is a bijection
\begin{equation}
\label{e:bijofroots}
\Delta^{i,-}_{\frac12}(\gamma)\leftrightarrow
\Delta^{r,-}_{\frac12}(\gamma').
\end{equation}
Fix an abstract Cartan subalgebra $\h_a$ of $\g$, and 
$\lambda_a\in\h_a^*$ giving the infinitesimal character of
$\caB$. Choose $\h'_a$ and $\lambda'_a$ for $\g'$ similarly. 
Pulling \eqref{e:bijofroots} back to $\h_a$ and $\h'_a$ as usual we
obtain a bijection
\begin{equation}
\Delta^{i,-}_{\frac12}(\gamma)_a\leftrightarrow
\Delta^{r,-}_{\frac12}(\gamma')_a.
\end{equation}
By Lemma \ref{l:cayleybasic} we obtain a bijection
\begin{equation}
\label{e:sbij}
S_i(\gamma)_a
\leftrightarrow S_r(\gamma')_a
\end{equation}
which we will denote by 
$\frs \mapsto \frs'$.

Pulling back to $\h_a$ and $\h'_a$, \eqref{e:assumptions}(b) gives bijections
\begin{equation}
\label{e:wisom}
\Delta(\lambda_a)\simeq \Delta(\lambda'_a),\quad
W(\lambda_a)\simeq W_a(\lambda_a'),
\end{equation}
which we denote
\begin{equation}
\label{e:deltaisom}
\alpha\mapsto\alpha',\quad w\mapsto w'.
\end{equation}

It is clear from \eqref{e:assumptions}(a) that the isomorphism 
$W(\lambda_a)\simeq W(\lambda_a')$
interchanges
$W^i(\lambda_a)$ and $W^r(\lambda'_a)$,
$W^r(\lambda_a)$ and $W^i(\lambda_a')$, and takes $W^C(\lambda_a)^\theta$ to
$W^C(\lambda'_a)^{\theta'}$.
Thus 
Proposition \ref{p:crossstab} implies
$W(\gamma)_a\simeq W(\gamma')_a$.
Therefore there is a natural isomorphism
\begin{equation}
\label{e:isom}
(S_i(\gamma)_a\times W(\lambda_a))/W(\gamma)_a\leftrightarrow
(S_r(\gamma)_a\times W(\lambda_a))/W_a(\gamma').
\end{equation}

The following result is an immediate consequence of this and
Proposition \ref{p:standardblocks}.  Write $\psi_i,\psi'_r$ for the
maps of \eqref{e:psis} applied to $\ol\scB$ and $\ol\scB'$, respectively.

\begin{proposition}
\label{p:standardbij}
In the 
setting  of Theorem \ref{t:gendual},
let $\scB=\scB(\gamma), \scB'=\scB(\gamma')$.
Recall (Section \ref{s:regularCharacters}) 
$\ol\scB=\scB/K$, and $\ol\scB'$ similarly. 
Using  bijections  \eqref{e:sbij} and \eqref{e:deltaisom},
the map
\begin{align*}
\Psi \; : \; \ol\scB &\longrightarrow \ol\scB' \\
\psi_i(\frs,w) &\mapsto \psi'_r(\frs',w')
\end{align*}
is a bijection.  By Theorem \ref{t:lang} this gives
a bijection $\Phi : \caB \rightarrow \caB'$.
Furthermore $\Psi$ commutes with the cross action and Cayley
transforms in the following sense. 
Fix $\delta\in\scB$.
\begin{enumerate}
\item[(a)]
$\Psi(c^\alpha(cl(\delta)))=
c_{\alpha'}(\Psi(cl(\delta)))$ for all $\alpha\in\Delta^{i,-}_{\frac12}(\delta)$,

\item[(b)]
$\Psi(c_\alpha(cl(\delta)))=c^{\alpha'}(\Psi(cl(\delta)))$
for all $\alpha\in\Delta^{r,-}_{\frac12}(\delta)$,

\item[(c)]
$\Psi(cl(w\times\delta))=cl(w'\times\Psi(cl(\delta)))$ for all $w\in
W_a(\lambda)$. 
\end{enumerate}
\end{proposition}

The proposition establishes most of the key properties needed to prove
that $\Psi$ satisfies the conditions of Definition \ref{d:vogandual}. The
Kazhdan-Lusztig algorithm of \cite{rt3} 
does not work within a fixed block of representations, but
instead with the larger set $\bfB$ of \eqref{e:B}.
So we need to extend the bijection of Proposition
\ref{p:standardbij}.

By \eqref{e:assumptions}(c) $W_{\frac12}(\lambda)\simeq
W_{\frac12}(\lambda')$.
With $\lambda_a$ and $\lambda'_a$ as in the discussion following the
theorem we also have
$W_{\frac12}(\lambda_a)\simeq
W_{\frac12}(\lambda'_a)$. 
Choose $\caF$ containing $\lambda_a$ as  in Section \ref{s:ca}.
Use this to define the
cross action of $W_a$, and to define $\bfB$ as in \eqref{e:B}.
Choose $\caF'$ for $\g'$, and define $\bfB'$, similarly.
The analogue of Proposition \ref{p:standardbij} for $\bfB$ follows
from Proposition \ref{p:standardblocks2}.

\begin{proposition}
\label{p:standardbij2}
Retain the hypotheses and notation of Theorem \ref{t:gendual}. 
With $\ol \bfB$ and $\ol\bfB'$ denoting $K$-conjugacy classes
in $\bfB$ and $\bfB'$ as above,
there is a bijection
\[
\Psi \; : \; \ol \bfB \longrightarrow \ol \bfB'
\]
extending the bijection of Proposition \ref{p:standardbij}.
Furthermore properties (a-c) of Proposition \ref{p:standardbij} hold,
with $W_{\frac12}(\lambda_a)$ in place of  $W(\lambda_a)$ in (c).
\end{proposition}

\begin{proof}[Proof of Theorem \ref{t:gendual}]
We are now in a position to work within the extended Hecke algebra
formalism of \cite[Section 9]{rt3}.  Proposition 9.5
of \cite{rt3} defines an algebra $\caH$ over the ring of formal
Laurent polynomials $\bbZ[q,q^{-1}]$ whose structure depends
only on $\Delta(\lam)$ and $\Delta_\frac12(\lam)$.  (By definition,
$\caH$ contains the Hecke
algebra of the integral Weyl group $W(\lam)$.)
The operators defined in
\cite[Definition 9.4]{rt3}, generalizing those of \cite[Definition 12.4]{ic4}, 
give rise to an $\caH$ module $\caM$ with a $\bbZ[q,q^{-1}]$ basis 
$\{m_\gamma \; | \; \gamma \in \ol \bfB\}$.
Since $\Delta(\lam)$ and $\Delta_\frac12(\lam)$ identify with
$\Delta(\lam')$ and $\Delta_\frac12(\lam')$ by the hypotheses of \eqref{e:phiWmatch} and
\eqref{e:phiWmatch2}, we also
obtain an $\caH$ module $\caM'$ with
$\bbZ[q,q^{-1}]$ basis $\{m'_{\gamma'} \; | \; \gamma' \in \ol \bfB'\}$.

We define an $\caH$ 
module structure on $\caM^*$, the $\bbZ[q,q^{-1}]$ linear dual of $\caM$,
as in \cite[Definition 13.3]{ic4} or
\cite[Equation 11.3]{rt3}.  
(Since $\caH$ is nonabelian, a little care is
required.)
The dual module $\caM^*$ is equipped
with the basis $\{\mu_{\gamma}\}$ dual to $\{m_\gamma\}$.
Write the bijection
$\Psi: \ol \bfB \rightarrow \ol \bfB'$ 
of Proposition \ref{p:standardbij2} as $\gamma \mapsto \gamma'$.  It
thus gives a $\bbZ[q,q^{-1}]$-linear isomorphism
\begin{align}
\label{e:heckemap}
\caM^* &\longrightarrow \caM' \\
\mu_{\gamma} & \longrightarrow (-1)^{l(\gamma)}m'_{\gamma'},
\end{align}
where $l(\gamma)$ is described in Remark \ref{r:length}.
As in the proofs of \cite[Theorem 13.13]{ic4} and \cite[Theorem
11.1]{rt3}, the existence of the duality
between elements of $\b$ and $\b'$ follows from the assertion that the
map of \eqref{e:heckemap} is in fact an $\caH$ module
isomorphism.  Again like the proofs of \cite[Theorem 13.13]{ic4} and
\cite[Theorem 11.1]{rt3}, this follows formally from the definition of
the $\caH$-module structure given in \cite[Definition 9.4]{rt3} and
the key symmetry properties summarized in parts (a)--(c)
of Proposition \ref{p:standardbij2}.
(As an example of the formal calculations involved
one may consult the proof of \cite[Theorem 11.1]{rt3}.)  This completes
the proof.
\end{proof}

\medskip

For use in the next section we need the concept of isomorphism of
blocks. 

Suppose $\tG$ is a two-fold cover of a real reductive linear group, and
$\caB$ is a block of genuine representations of $\tG$ with regular
infinitesimal character. Let $\scB$ be the corresponding block of
genuine regular characters. 
Fix a family $\caF$ as in Section \ref{s:F},
and use it to define $\bfB$ as in \eqref{e:B}.
As in the proof of Theorem \ref{t:gendual}, the set  $\ol{\bfB}$
of 
$K$-orbits 
in $\bfB$ index
a basis of a module $\caM$ for the extended Hecke algebra $\caH$.
It is easy to see that that the structure of $\caM$ as
a based module for $\caH$ does not depend on the choice of $\caF$.

Now suppose $\caB_1,\caB_2$ are blocks of genuine representations of
groups $\tG_1,\tG_2$ as in the preceding paragraph.
Write $\lambda_i$ for the infinitesimal character of $\caB_i$, and
assume we are given an isomorphism $\Delta_{\frac12}(\lambda_1)\simeq
\Delta_{\frac12}(\lambda_2)$ taking
$\Delta(\lambda_1)$ to 
$\Delta(\lambda_2)$, and $\caH_1\simeq \caH_2$.
For $i=1,2$ let $\bfB_i$, $\caM_i$ and $\caH_i$ be as above.

\begin{definitionplain}
\label{d:blockisom}

We say $\caB_1$ is {\it isomorphic} to $\caB_2$ if
there is a bijection $\ol{\bfB}_1\rightarrow \ol{\bfB}_2$ which induces
an isomorphism $\caM_1\simeq \caM_2$ as $\caH_1\simeq\caH_2$ modules.
\end{definitionplain}

As in the discussion after \eqref{e:heckemap} an isomorphism of blocks
preserves multiplicity matrices:

\begin{lemma}
\label{l:isommult}
Suppose $\scB_1\simeq\scB_2$.
Writing the isomorphism
$\gamma_1\rightarrow\gamma_2$, for all $\delta_1,\gamma_1\in\scB_1$ we have:
\begin{equation}
\begin{aligned}
m(\delta_1,\gamma_1)&=m(\delta_2,\gamma_2),\\
M(\delta_1,\gamma_1)&=M(\delta_2,\gamma_2).\\
\end{aligned}
\end{equation}
\end{lemma}

As in the proof of Theorem \ref{t:gendual}, rigidity of genuine
representations gives the following result.

\begin{proposition}
\label{p:isomorphic}
In the setting of Definition \ref{d:blockisom},
suppose $\gamma_i\in\scB_i$ are weakly minimal, and
$\ghalf(\gamma_1)\simeq \ghalf(\gamma_2)$ (Definition 
\ref{d:big}).
Then $\b_1$ is isomorphic to $\b_2$ in the sense of Definition
\ref{d:blockisom}.
\end{proposition}

\noindent{\bf Sketch.}  
Since $\ghalf(\gamma_1)\simeq \ghalf(\gamma_2)$,
Proposition \ref{p:standardblocks2} parametrizes
$\ol{\bfB}_1$ and $\ol{\bfB}_2$ in terms of the same set, and gives
a bijection between them.  To conclude $\caB_1$ and $\caB_2$
are isomorphic, we need to verify that the induced map
$\caM_1 \rightarrow \caM_2$ is an isomorphism of $\caH_1 \simeq \caH_2$
modules.  This follows from the formulas
of \cite[Proposition 9.4]{rt3} defining the $\caH_i$ module
structure.
\qed

\sec{Definition of the dual regular character}
\label{s:defofdual}

Theorem \ref{t:gendual} reduces the duality for a block $\caB$ to the
existence of a single regular character with prescribed properties.
The point of this section is to construct that character.  This may be
done directly, but instead we choose to use results from
\cite{aherb}.  There is a good reason for doing so: 
as discussed in the introduction there is a
close connection between the results of the current paper and those
of \cite{aherb}.  We begin by describing this relationship.

\medskip

Suppose for the moment that $G$ is the real points of a connected
complex group $G_\C$, and $\caB$ is a block of irreducible
representations of $G$ with regular infinitesimal character.  As in
\cite{ic4} let $\ch \caB$ be a dual block for $\ch G$, a real form of
$\ch G_\C$. Write $\pi\rightarrow\ch\pi$ for the duality map on the
level of irreducible representations.  Define an equivalence relation
on elements of $\caB$ by $\pi \sim \eta$ if $\supp(\ch \pi) =
\supp(\ch \eta)$, where $\supp$ is discussed before Proposition
\ref{p:geom}.  Equivalence classes for this relation are
classical L-packets for $G$ \cite[Section 15]{ic4}.  From this point
of view, they are parametrized by a subset of the orbits of the
complexification, say $U_\C$, of the maximal compact subgroup of $\ch
G$ on the flag variety $\ch X$ for $\ch\g$.

Now fix a block of genuine representations $\wt \caB$ for a
nonlinear cover $\tH$ of a linear group $H$.
Suppose, for example in the setting of Theorem
\ref{t:mainintro}, there is group $H'$, with (typically nonlinear)
cover $\tH'$, and a block of genuine representations of $\tH'$ which
is dual to $\wt\caB$. Let $\tK'$ be a maximal compact subgroup of
$\tH'$, with complexification $\tK'_\C$.  Then we may define an
equivalence relation on $\wt\caB$ in the same way as in the previous
paragraph, to partition $\wt\caB$ into subsets, analogous to
L-packets for a linear group. As before, these are parametrized by the orbits of
$\tK'_\C$ on the  flag variety $X'$ for $\h'$.
(In spite of this analogy we do not refer to these subset as L-packets
for $\tH'$. See \cite[Section 19]{aherb}.)
Note that Proposition \ref{p:geom} implies that these
subsets are singletons if $\wt H$ is simply laced. 

Let $K'_\C$ be the complexification of a maximal compact subgroup of
$H'$. There is a surjection $\tK'_\C\rightarrow K'_\C$ with central
kernel; it follows that the orbits of $\tK'_\C$ and $K'_\C$ on $X'$
are the same. Now suppose $H'=\ch G$, the group appearing in Vogan
duality for the linear group $G$.
We conclude that the L-packets
in $\caB$ and the subsets  of $\caB'$ defined above are in natural bijection,
parametrized by $\tK'_\C$ or $U_\C$ orbits on the flag variety for
$\ch\g=\h'$. 
Examples of this phenomenon include  \cite{rt3}:
$$
G=GL(n,\R), \ch G=U(p,q), \tH=\wt{GL}(n,\R), \tH'=\wt U(p,q)
$$
and \cite{rt2}:
and
$$
G=SO(2p,2q+1), \ch G=Sp(2n,\R), \tH=\wt{Sp}(2n,\R), \tH'=\wt{Sp}(2n,\R).
$$

According to Langlands and Shelstad, associated to
each packet $\Pi$ in $\caB$
is a certain interesting {\it stable} virtual character $\Theta_\Pi$.
(For tempered L-packets this is the sum of the representations in the
packet.) This sum can be defined as in \cite[Definition 1.28]{abv} using Vogan
duality $\caB\leftrightarrow \ch\caB$. Suppose $\wt\Pi$ is the
corresponding subset of $\wt\caB'$. Definition 1.28 of 
\cite{abv} applies in this situation to give 
a distinguished genuine  virtual character $\Theta_{\wt\Pi}$ of
$\tH'$, although 
(since the notion of stability 
is not defined for $\tH'$)
it is not clear what properties it should have. In any event,
since duality is closely related to character theory, it is reasonable
to expect that the map $\Theta_\Pi\rightarrow \Theta_{\wt\Pi}$ has
nice properties.
In fact in the simply laced case \cite{aherb}, and for $\tH'=\wt{Sp}$
\cite{adams:lifting},  there is a theory of {\it lifting of characters}
which takes $\Theta_\Pi$ to $\Theta_{\wt\Pi'}$.

This reasoning may  be turned around: a theory of lifting of
characters can be used to give information about Vogan duality. 
This is what happens here:
we use some parts of the theory of \cite{aherb} to solve some
technical issues arising here in defining duality of characters.
We turn to this now.

\bigskip

Fix an admissible two-fold cover $\tG$
of a simply laced, real reductive linear group $G$
(cf.~ Sections \ref{s:notation} and \ref{s:nonlinear}),
and a block $\scB$ of genuine regular characters
of $\tG$.
Fix a weakly minimal element $\gamma$ of $\scB$, and let
$\b=\b(\irr(\gamma))$ be the corresponding block of representations. 
We will construct a group $G'$, with admissible cover $\tG'$, and a
genuine regular character $\gamma'$ of $\tG'$ 
such that Theorem \ref{t:gendual} applies to prove
that $\b'=\b(\irr(\gamma'))$ 
is  dual to $\b$.

Since we will be passing back and forth between linear and nonlinear
groups
 we change notation and write
$\wt\gamma=(\tH,\wt\Gamma,\wt\lambda)$ for a genuine regular character
of $\tG$. Let $\b=\b(\irr(\wt\gamma))$. 

We first make an elementary reduction. Recall a connected complex
group $G_\C$ is said to be {\it acceptable} if (for any Cartan
subgroup $H_\C$ and set of positive roots) one-half the sum of the
positive roots exponentiates to $H_\C$. 

\begin{lemma}
\label{l:Gd0}
There is a connected, reductive complex group $G'_\C$, with acceptable derived group,
real form $G'$, admissible cover $\tG'$, and block 
$\b'$ of genuine representations of $\tG'$ such that $\b'$ is
isomorphic to $\b$  (Definition \ref{d:blockisom}).
\end{lemma}

\begin{proof}
Thanks to Proposition \ref{p:isomorphic} this isn't hard, and there are various ways to construct
$G'$. Here is one.

Let $G^{sc}_\C$ be the simply connected cover of the derived group of
$G_\C$, with real points $G^{sc}$. Then $G^{sc}$ maps onto the
identity component $G_d^0$ of the derived group $G$; let $\tG^{sc}$ be the pullback of
the restriction of the cover $\tG\rightarrow G$ to $G_d^0$. This is
the admissible cover of $G^{sc}$. Let $\wt H_d=\tH\cap \tG^0_d$, and let
$\tH^{sc}$ be the inverse image of this in $\tG^{sc}$.  

Let $\wt\lambda^{sc}$ be the restriction of $\wt\lambda$ to the Lie
algebra of $\tH^{sc}$.
Let $\wt\Gamma_d$ be an irreducible representation of $\wt{H_d}$
contained in
the restriction of $\wt\Gamma$ to $\tH_d$, and let $\wt\Gamma^{sc}$ be
the pullback of this to $\tH^{sc}$.

It is easy to check that
$\wt\gamma^{sc}=(\tH^{sc},\wt\Gamma^{sc},\wt\lambda^{sc})$ is
a genuine regular character of $\tG^{sc}$. It is also clear that it
has the same bigrading as $\wt\gamma$ (cf.~Proposition
\ref{p:gradings}). It follows from Proposition \ref{p:isomorphic}  that
$\b(\irr(\wt\gamma^{sc}))$ is isomorphic to $\b(\irr(\wt\gamma))$.
\end{proof}

Therefore, after replacing $G$ by $G'$ if necessary, we assume  $G$ is
the real points of $G_\C$, where $G_\C$ is connected reductive and 
the derived group of $G_\C$ is acceptable.

We recall some notation from \cite{aherb}.  We consider the {\it
admissible triple} $(\tG,G,G)$ \cite[Definition 3.14]{aherb}.
This amounts to saying that 
$G_\C$ is simply laced, with acceptable derived group, and 
$\tG$ is an admissible cover of $G$.  We fix a set of
lifting data for this triple \cite[Definition 7.1]{aherb}.
We don't need to spell out  this construction, including the choices
involved, but merely summarize the properties that we need.

Associated to this data is a character $\mu$ of the center
$\mathfrak z$ of $\g$ \cite[Definition 6.35]{aherb}.

Let $H$ be the Cartan subgroup of $G$ corresponding to $\tH$.
Lifting data for $(\tH,H,H)$  is chosen as in \cite[Section
17]{aherb}.
If $\Gamma$ is a character of $H$ then $\Lift_H^{\tH}(\Gamma)$ 
is $0$ or a sum of irreducible representations of $\tH$. 
See \cite[Section 10]{aherb}.

Lifting of character data is  defined as follows. 
Suppose $\gamma=(H,\Gamma,\lambda)$ is a regular character of $G$.
Write
\begin{equation}
\label{e:liftHnew}
\Lift_H^{\tH}(\Gamma e^{-2\rho_i+2\rho_{i,c}})
=\sum_i\wt\Gamma_ie^{-2\rho_i+2\rho_{i,c}}
\end{equation}
Here $\rho_i=\rho_i(\lambda),\rho_{i,c}(\lambda)$ are as in Definition \ref{d:reg}, 
and occur here due to the difference between character
data (cf. Section \ref{s:regularCharacters}) and {\it modified}
character data of \cite[Section 16]{aherb}.
By \cite[Section 17]{aherb} and \eqref{e:liftHnew} we have
\begin{equation}
2d\wt\Gamma_i=d\Gamma+\rho_i-2\rho_{i,c}-\mu.
\end{equation}
Then
$\Lift_G^{\tG}=\{(\tH,\wt\Gamma_i,\wt\lambda)\,|\, 1\le i\le n\}$
where $\wt\lambda=\frac12(\lambda-\mu)$,
and $\wt\Gamma_i$ are given by \eqref{e:liftHnew}.

An important role is played by the character $\zeta_{cx}$ of
\cite[Section 2]{aherb}. 
Let $\Gamma_r(H)$ be the subgroup of $H$  generated by the $m_\alpha$ for all
real roots $\alpha$. 
Let $G_d$ be the derived group of $G$, and let $H_d^0$ be the identity
component of $H_d=H\cap G_d$.
It is well-known that $H_d=\Gamma_r(H)H_d^0$.
Let $S$ be a set of complex roots of $H$ such that the set of all
complex roots is $\{\pm\alpha,\pm\theta\alpha\,|\,\alpha\in S\}$. Let
$\zeta_{cx}(h)=\prod_{\alpha\in S}\alpha(h)$ $(h\in \Gamma_r(H))$.
Then $\zeta_{cx}$ is independent of the choice of $S$.

Fix a genuine regular character $\wt\gamma$ of $\tG$. We do not yet
need to assume it is weakly minimal.

\begin{lemma}
\label{l:defineGamma}
There is a character $\Gamma$ of $H$ such that
$\gamma=(H,\Gamma,2\wt\lambda+\mu)$ is a regular character, and
$\wt\gamma\in\Lift_G^{\tG}(\gamma)$.
\end{lemma}
This is an immediate consequence of \cite[Lemma 17.13]{aherb}.
In fact we only need the following weaker statement, whose proof is
fairly self-contained.

Choose a set of positive roots satisfying:
\begin{subequations}
\renewcommand{\theequation}{\theparentequation)(\alph{equation}}  
\label{e:special}
\begin{align}
\alpha\text{ complex},\,\alpha>0&\Rightarrow \theta(\alpha)<0,\\
\alpha\text{ imaginary},\, \alpha>0&\Rightarrow \langle\lambda,\ch\alpha\rangle>0.
\end{align}
\end{subequations}
Define $\rho,\rho_i$ and $\rho_{i,c}$ accordingly. 

Define:
\begin{subequations}
\renewcommand{\theequation}{\theparentequation)(\alph{equation}}  
\label{e:defineGamma}
\begin{align}
\Gamma(h)&=(\wt\Gamma^2e^{\rho})(h)e^{-2\rho_i+2\rho_{i,c}}(h)|e^{\rho_i-\rho}(h)|
\quad(h\in H^0_d),\\
\Gamma(h)&=\zeta_{cx}(h)\quad(h\in \Gamma_r(H)).
\end{align}
\end{subequations}
To be precise in (a), since the derived group  is acceptable $e^{\rho}(h)$ is well
defined, and clearly $e^{-2\rho_i+2\rho_{i,c}}(h)$ 
and $|e^{\rho_i-\rho}(h)|=|e^{2\rho_i-2\rho}(h)|^{\frac12}$ are 
well defined
since the exponents are sums of roots.

\begin{lemma}
\label{l:defineGamma2}
There is a regular character $\gamma=(H,\Gamma,\lambda)$ where
$\Gamma$ restricted to $H_d$ is given by \eqref{e:defineGamma}, and 
$\lambda$ satisfies
$\langle\lambda,\ch\alpha\rangle=\langle2\wt\lambda,\ch\alpha\rangle$
for all roots $\alpha$.   
\end{lemma}

\begin{proof}
By \cite[(6.21)]{aherb}  \eqref{e:defineGamma}(a) and (b) agree on
$\Gamma_r(H)\cap H_d^0$,
and so define $\Gamma|_{H_d}$.
Choose any extension of $\Gamma$ to $H$ and set
$\lambda=d\Gamma-\rho_i(\wt\lambda)+2\rho_{i,c}(\wt\lambda)$
(cf.~\eqref{e:regchar2}). 
It is easy to check that
$(H,\Gamma,\lambda)$ is a regular character, and that $\lambda$ and
$2\wt\lambda$ have the same restriction to $\h_d$.
\end{proof}

The proof of Lemma \ref{l:defineGamma} in \cite{aherb} is the same, except that 
we are more careful in choosing  the extension of $\Gamma$ to $Z(G)^0H_d$, so that
$\lambda=2\wt\lambda+\mu$ where $\mu\in\mathfrak z^*$ is given. This
isn't necessary for our purposes.

A crucial property of $\gamma$ is the following.

\begin{proposition}
\label{p:crucial}
Fix $\alpha\in \Delta(2\wt\lambda)$.

\medskip\noindent (a)
Suppose $\alpha$ is imaginary. Then $\alpha$ is compact if and only if
$\langle\wt\lambda,\ch\alpha\rangle\in\Z$. 

\smallskip\noindent (b)  Suppose $\alpha$ is real. Then $\alpha$
does not satisfy the parity condition with respect to $\gamma$ if and
only if $\langle\wt\lambda,\ch\alpha\rangle\in\Z$. 

\medskip
In other words
\begin{subequations}
\renewcommand{\theequation}{\theparentequation)(\alph{equation}}  
\begin{align}
\label{e:Deltar}
\Delta^{r,+}(\gamma)&=\Delta^r(\wt\lambda)\\
\label{e:Deltai}
\Delta^{i,+}(\gamma)&=\Delta^i(\wt\lambda).
\end{align}
\end{subequations}
\end{proposition}

\begin{proof}
If $\alpha$ is imaginary the assertion follows immediately from
\eqref{e:gradings}(b), so assume $\alpha$ is real.
We need to show (see Section \ref{s:bigradings}):
\begin{subequations}
\label{e:Gammamalpha}
\renewcommand{\theequation}{\theparentequation)(\alph{equation}}  
\begin{equation}
\Gamma(m_\alpha)=
\begin{cases}
\epsilon_\alpha(-1)^{\langle 2\wt\lambda,\ch\alpha}\rangle&\langle\wt\lambda,\ch\alpha\rangle\in\Z\\  
-\epsilon_\alpha(-1)^{\langle 2\wt\lambda,\ch\alpha}\rangle&\langle\wt\lambda,\ch\alpha\rangle\in\Z+\frac12,
\end{cases}
\end{equation}
i.e.
\begin{equation}
\Gamma(m_\alpha)=\epsilon_\alpha.
\end{equation}
\end{subequations}

\begin{lemma}
\label{l:zetacx}
We have $\zeta_{cx}(m_\alpha)=\epsilon_\alpha$.
\end{lemma}

\begin{proof}
Let
$d=\sum_X\langle\beta,\ch\alpha\rangle$
where $X$ is the set of positive  roots which become non-compact imaginary roots
on the Cayley transform $H_\alpha$ (cf.~Section \ref{s:ca}).
According to (\cite{green}, Lemma 8.3.9)
$\epsilon_\alpha=(-1)^d$.
A root $\beta$ of $H$ is imaginary for $H_\alpha$ if
$s_\alpha\theta\beta=\beta$, in which case $s_\alpha\beta$ is also
imaginary for $H_\alpha$.
If $\langle\beta,\ch\alpha\rangle\ne 0$  is even
$\beta$ and $s_\alpha\beta$ are both compact or both non-compact for
$H_\alpha$, otherwise one is compact and one is non-compact. Therefore
the contribution of $\{\pm\beta,\pm\theta\beta\}$ to $d$ is
$\langle\beta,\ch\alpha\rangle\pmod 2$, and we can replace $X$ with
\begin{equation}
\label{e:X}
\{\beta\in S\,|\,
s_\alpha\beta=\theta\beta,\langle\beta,\ch\alpha\rangle\text{
  odd}\}.
\end{equation}

On the other hand  $\zeta_{cx}(m_\alpha)=(-1)^e$ where 
$e=\sum_S\langle\beta,\ch\alpha\rangle$.
If $\langle\beta,\ch\alpha\rangle=0$ then $\beta$ does not contribute
to the sum; if $s_\alpha\beta\ne\pm\theta\beta$ the total contribution
of $\beta,s_\alpha\beta$ is even. It is easy to see the case
$s_\alpha\beta=-\theta\beta$ does not arise.
Therefore
we can replace $S$ with the set \eqref{e:X}.
\end{proof}

The proposition now follows from the Lemma,  \eqref{e:defineGamma}(b) and
\eqref{e:Gammamalpha}(b). 
\end{proof}

We now want to construct a group $G'$ with the following properties:

\begin{enumerate}
\item[(1)] $G'$ is the real points of a connected, complex reductive group $G'_\C$,
\item[(2)] There is a regular character $\gamma'$ for $G'$ such that the
  bigrading of $\gamma'$ is dual to that of $\gamma$,
\item[(3)] the derived group of $G'_\C$ is simply connected,
\item[(4)] $G'$ admits an acceptable cover.
\end{enumerate}

The main point is (2), and is provided by duality for $G$ \cite{ic4}.
A little tinkering may be required to satisfy conditions (3) and (4).

We  first construct a connected, reductive complex group $G'_\C$, with real
points $G'$, and a regular character $\gamma'$ of $G'$, such that 
$\b(\irr(\gamma'))$ is dual to
$\b(\irr(\gamma))$, with $\irr(\gamma)$ going to $\irr(\gamma')$. This
essentially follows from \cite[Theorem 13.13]{ic4}, although the group
$G'$ constructed there is not necessarily the real points of a
connected complex group. That we can construct such a group $G'$
follows from \cite[Theorem 1.24 and Chapter 21]{abv}. For the case of
integral infinitesimal character see \cite{bowdoin}.

Thus condition (1) holds, and (2) holds as in 
\cite[Theorem 11.1]{ic4}.

It is well-known that we can choose a connected, reductive complex group
$G''_\C$, with real points $G''$, with simply connected derived group,
such that there is a surjection $G''_\C\rightarrow G'_\C$
taking $G''$ onto $G'$. Pulling back $\irr(\gamma')$ to $G''$ we
obtain an isomorphic block for $G''$. After making this change 
we may assume conditions (1-3) hold.

If $G'$ is semisimple then condition (4) is automatic. For $G'$
reductive this condition may fail (although this is unusual). If this
is the case, replace $G'_\C$ with its derived group, with real points
$G'_d$, and $\gamma'$ with its restriction to $G'_d$. This regular
character has the same bigrading, and now conditions (1-4) hold. 

So we now assume we are given $G'$ and $\gamma'$ satisfying (1-4).
Write $\gamma'=(H',\Gamma',\lambda')$. We spell out the essential
condition (2) (cf.~\cite[Theorem 11.1]{ic4}).

There is an isomorphism $\phi:\h'\rightarrow
\C\langle\Delta(2\lambda)^\vee\rangle\subset \h$, satisfying
\begin{subequations}
\label{e:phi}
\renewcommand{\theequation}{\theparentequation)(\alph{equation}}  
\begin{equation}
\phi(\ch\Delta(\g',\h'))=\ch\Delta(2\wt\lambda).
\end{equation}
This satisfies
\begin{equation}
\phi^*(\theta(\alpha))=-\theta'(\phi^*(\alpha))
\end{equation}
where $\theta'$ is the Cartan involution for $G'$.
Furthermore $\lambda'$ is integral for $\Delta(\g',\h')$. Since the
derived group of $G'_\C$ is simply connected we may assume
\begin{equation}
\label{e:translate}
\langle 2\wt\lambda,\ch\alpha\rangle = 
\langle \lambda',\phi^*(\alpha)^\vee\rangle
\end{equation}
for all $\alpha\in\Delta(2\wt\lambda)$. 
Finally we have
\begin{align}
\phi^*(\Delta^{r,+}(\gamma))&=\Delta^{i,+}(\gamma'),\\
\phi^*(\Delta^{i,+}(\gamma))&=\Delta^{r,+}(\gamma').
\end{align}
\end{subequations}

Now let $\tG'$ be an admissible cover of $G'$.
We want to lift $\gamma'$ to a genuine regular character of $\tG'$. 

\begin{proposition}
There is a choice of transfer data  for $\tG'$ (as in \cite{aherb}) so that
$\Lift_{G'}^{\tG'}(\gamma')$ is nonzero.  
\end{proposition}

\begin{proof}
This follows immediately from \cite[Lemma 17.14]{aherb}, the
hypothesis of which holds by the following key result.
\end{proof}

\begin{lemma}
\label{l:key}
Let $\wt\chi$ be any genuine character of $\tH'$. Choose positive
roots satisfying \eqref{e:special} with respect to $\lambda'$, 
and define  $\rho',\rho'_i$ and $\rho'_{i,c}$ accordingly.
Suppose $h\in H_d$, $h^2=1$. 
Then
\begin{equation}
\Gamma'(h)e^{-2\rho'_i+2\rho'_{i,c}}(h)=
\begin{cases}
(\wt\chi^2e^{\rho'})(h)&(h\in H_d^{'0}),\\
\zeta_{cx}(h)&(h\in \Gamma_r(H')).
\end{cases}
\end{equation}
\end{lemma}

\begin{proof}
First suppose $\beta\in\Delta(\g',\h')$ is a real root, so
$m_{\beta}\in \Gamma_r(H')$. 
We need to show $\Gamma'(m_{\beta})=\zeta_{cx}(m_{\beta})$. By
Lemma \ref{l:zetacx}
we can write this as
\begin{equation}
\Gamma'(m_{\beta})=
\begin{cases}
\epsilon_\beta(-1)^{\langle \lambda',\ch\beta\rangle}&  \langle
\lambda',\ch\beta\rangle\in 2\Z\\
-\epsilon_\beta(-1)^{\langle \lambda',\ch\beta\rangle}&  \langle
\lambda',\ch\beta\rangle\in 2\Z+1.
\end{cases}
\end{equation}
This says
\begin{equation}
\beta\in\Delta^{r,+}(\gamma')\Leftrightarrow
\langle\lambda',\ch\beta\rangle\in 2\Z.
\end{equation}
By \eqref{e:translate} if we let
$\alpha=\phi^{*-1}(\beta)\in\Delta(2\wt\lambda)$
we can write this as 
\begin{equation}
\beta\in\Delta^{r,+}(\gamma')\Leftrightarrow
\langle\wt\lambda,\ch\alpha\rangle\in \Z.
\end{equation}
By \eqref{e:phi}(b) $\alpha$ is imaginary, and by 
by \eqref{e:Deltai} equality holds on the right hand side 
if and only if $\alpha\in\Delta^{i,+}(\gamma)$, so we need to show
\begin{equation}
\phi^*(\Delta^{i,+}(\gamma))=\Delta^{r,+}(\gamma')
\end{equation}
which is precisely 
\eqref{e:phi}(e).

Now suppose $Z'\in i\mathfrak h'_0$, $\exp(2\pi i Z')=1$, and $h'=\exp(\pi
iZ')\ne1$. 
We have to show
\begin{equation}
\Gamma'(h')e^{-2\rho'_i+2\rho'_{i,c}}(h')=(\wt\chi^2e^{\rho'})(h').
\end{equation}
The left hand side is
$e^{\pi i\langle d\Gamma'-2\rho'_i+2\rho'_{i,c},Z'\rangle}$.
Using \eqref{e:regchar4} the exponent is 
$\pi i\langle \lambda'-\rho'_i,Z'\rangle$, so we have to
show
\begin{equation}
e^{\pi i\langle \lambda'-\rho'_i,Z'\rangle}=
(\wt\chi^2e^{\rho'})(h').
\end{equation}
Write $\rho'=\rho'_r+\rho'_i+\rho'_{cx}$. 
Then $\langle \rho'_r,Z'\rangle=0$, and 
$\langle \rho'_{cx},Z'\rangle=0$ by \eqref{e:special}. Therefore 
we can replace $\rho'_i$ with $\rho'$ on the left hand side, and we 
are reduced to showing
\begin{equation}
e^{\pi i\langle \lambda',Z'}\rangle=\wt\chi^2(h').
\end{equation}

Write $\wt{\exp}$ for the exponential map $\h'_0\rightarrow \tH'$. 
Then the right hand
side is  $\wt\chi^2(\exp(\pi iZ'))=\wt\chi(\wt{\exp}(2\pi iZ'))$, so
we have to show
\begin{equation}
\label{e:toshow}
e^{\pi i\langle \lambda',Z'\rangle}=\wt\chi(\wt{\exp}(2\pi iZ')).
\end{equation}

Suppose $\beta$ is an imaginary root. Since 
the derived group of $G'_{\C}$ is simply connected 
we can take
$Z'=\ch\beta$. By Lemma \ref{l:wtexp}(a) the right hand side is $1$
if $\beta$ is compact, and $-1$ otherwise. In other words 
the right hand side is $1$ if and only if $\beta\in\Delta^{i,+}(\gamma')$.

By \eqref{e:phi}(a) $\alpha=\phi^{*-1}(\beta)$ is real, and 
by \eqref{e:translate} the left hand side is equal to $e^{2\pi
  i\langle\wt\lambda,\ch\alpha\rangle}$.
Therefore by \eqref{e:Deltar}  the left hand side is $1$ if and only
if $\alpha\in \Delta^{r,+}(\gamma)$.
So we have to show
$\phi^*(\Delta^{r,+}(\gamma))=\Delta^{i,+}(\gamma')$, which is 
\eqref{e:phi}(d).
This proves \eqref{e:toshow} in
this case.

Now suppose $\beta$ is a complex root, in which case we can take
$Z'=\ch\beta+\theta'\ch\beta$. 
We have to show
\begin{equation}
e^{\pi i\langle\lambda',\ch\beta+\theta'\ch\beta\rangle}
=
\wt\chi(\wt\exp(2\pi i(\ch\beta+\theta'\ch\beta)))
\end{equation}

By \eqref{e:phi}(b) and \eqref{e:translate} we can write the exponent on the left hand side
as 
\begin{equation}
\pi i \langle 2\wt\lambda, \ch\alpha-\theta\ch\alpha\rangle
=
\pi i \langle2\wt\lambda, \ch\alpha+\theta\ch\alpha\rangle
-2\pi i \langle 2\wt\lambda, \theta\ch\alpha\rangle.
\end{equation}
The final term is an integral multiple of $2\pi i$, so the left hand
side is
$e^{2\pi i\langle \wt\lambda,\ch\alpha+\theta\ch\alpha\rangle}$.
Now $\wt\lambda$ is the differential of a genuine character $\wt\tau$
of $\tH^0$ so we can write this as 
\begin{equation}
\wt\tau(\wt\exp(2\pi i(\ch\alpha+\theta\ch\alpha))).
\end{equation}
Therefore we need to show
\begin{equation}
\wt\tau(\wt\exp(2\pi i(\ch\alpha+\theta\ch\alpha)))=
\wt\chi(\wt\exp(2\pi i(\ch\beta+\theta'\ch\beta))).
\end{equation}
This follows from Lemma \ref{l:wtexp}(b): both sides are $1$ if and
only if $\langle \alpha,\theta\ch\alpha\rangle=0$.

Every element of $H_d^0$ of order $2$ is of the form $\exp(\pi iZ')$
for $Z'$ as above, so this completes the proof.
\end{proof}

Let $\wt\gamma'$ be any constituent of 
$\Lift_{G'}^{\tG'}(\gamma')$. We can write
$\wt\gamma'=(\tH',\wt\Gamma',\wt\lambda')$ for some $\wt\Gamma'$, and
$\wt\lambda'=\frac12(\lambda'-\mu')$. Here $\mu'$ is an element of the
dual of the center of $\g'$, depending on the choice of lifting data.
Recall by \eqref{e:translate} for any root $\alpha$ we have
$\langle 2\wt\lambda,\ch\alpha\rangle = \langle
\lambda',\phi^*(\alpha)^\vee\rangle$.
Therefore 
\begin{equation}
\label{e:wtlambda'}
\langle \wt\lambda,\ch\alpha\rangle = \langle
\wt\lambda',\phi^*(\alpha)^\vee\rangle.
\end{equation}
Therefore conditions \ref{e:assumptions}(b) and (c) of
Theorem \ref{t:gendual} hold. 

We now assume $\wt\gamma$ is weakly minimal.
Then the conditions of Theorem \ref{t:gendual} hold, and we
conclude $\b(\irr(\wt\gamma'))$ is dual to
$\b(\irr(\wt\gamma))$. 

We summarize the preceding discussion.

\begin{theorem}
\label{t:main}
Assume:

\begin{enumerate}
\item  $G$ is the real points of 
a connected, reductive, simply laced  complex group with acceptable
derived group, and $\tG$
is an admissible cover of $G$;
\item $\scB$ is a block of 
genuine regular characters of $\tG$, and $\wt\gamma$ is 
a weakly minimal element of $\scB$;
\item $\gamma$ is a  regular character of $G$ so that $\wt\gamma$ is
  a constituent of 
$Lift_G^{\tG}(\gamma)$ for some
choice of lifting data;
\item 
$G'$ is the real points of a connected, reductive complex group, with
simply connected derived group, and $\tG'$ is an admissible cover of
$G'$;
\item $\gamma'$ is a regular character of $G'$ such that $\gamma$
and $\gamma'$ have dual (weak) bigradings;
\end{enumerate}
Choose lifting data so
$\Lift_{G'}^{\tG'}(\gamma')$ is nonzero, and let $\wt\gamma'$ be any
constituent of this lift. Then $\b(\irr(\wt\gamma'))$ is dual to
$\b(\irr(\wt\gamma))$. 
\end{theorem}

\begin{remark}
It is clear from the proof that the assumptions are stronger than
necessary, and various weaker versions of the theorem are possible.   
\end{remark}

\begin{remark}
\label{rem:singular}
We have assumed that the infinitesimal character is regular.
There are several possibilities for formulating a version
of Definition \ref{d:vogandual} in the case of singular
infinitesimal character.  
This is explained in \cite{abv}, 
and we omit the details here.
\end{remark}

\sec{Examples}
\label{s:examples}

\subsec{Example: $SL(2,\R)$}
The unique nontrivial two-fold cover of $G$ is an admissible cover.
At infinitesimal character $\rho/2$ (a typical half-integral
infinitesimal character) there are four irreducible genuine
representations: the two oscillator representations, each of which has
two irreducible summands.  It is easy to see that at this
infinitesimal character  $\tG$ has two
genuine blocks with distinct central characters, each containing two
irreducible representations.  Each block is self-dual.
See \cite[Section 4]{rt1}.

\subsec{Example: $GL(n,\R)$} 
\label{s:gln}
A two-fold cover $\tG$ of $G = \GL(n,\R)$ is admissible if it
restricts to the (unique) nontrivial cover of $\SL(n,\R)$; assume this
is the case.  Suppose $\lambda$ is a half-integral infinitesimal
character.  Then $\Delta_{\frac12}(\lambda)$ is of type $A_p\times
A_q$.  If $n$ is even $\tG$ has a unique block at infinitesimal
character $\lam$; if $n$ is odd, there are two isomorphic blocks.  Fix
such a block $\b$.  Consider an admissible cover $\tG'$ of $G' =
\U(p,q)$; there are three, corresponding to the three double covers of
$\U(p) \times \U(q)$, and the exact one we choose is not important.
Fix an infinitesimal character $\lam'$ for which $\tG'$ has genuine
discrete series.  This implies that $\Delta_{\frac12}(\lambda)$ is of
type $A_p\times A_q$.  Then $\tG'$ has a unique block $\b'$ at
infinitesimal character $\lam'$, and $\b$ is dual to $\b'$.  

Two cases of the previous paragraph are worth keeping
in mind.  If $\lam$ is  integral  $\b$ consists of a single
irreducible principal series, $G'$ is compact, the cover $\tG'$
splits, and $\b'$ consists of a single finite-dimensional
representation of the compact linear group $\tG'$.  On the other hand,
consider $\lam = \rho/2$.  Then each block $\b$ for $\tG$ has an
interesting unitary representation, say $\pi$, that is as small as the
infinitesimal character permits.  (In other words, $\pi$ might be called
unipotent.)  The group $G'$ is either isomorphic to $\U(n/2,n/2)$ (if
$n$ is even) or $\U\left((n+1)/2, (n-1)/2\right)$ if $n$ is odd.  In
either case $\tG$ is quasisplit and hence has genuine large discrete
series.  Each such is characterized by its infinitesimal character.
Fix $\lam'$ for which such a large discrete series $\pi'$ exists.  Let
$\b'$ denote the block containing $\pi'$.  Then $\b$ is dual to
$\b'$ and the duality maps $\pi$ to $\pi'$.
For further details see
\cite[Part II]{rt3}.

\medskip
\subsec{Example: Minimal Principal Series of Split Groups}
\label{ex:split}
Suppose $G$ is the split real form of a connected, 
simply connected reductive complex group $G_\C$.
Suppose $\tG$ be 
an admissible cover of $G$ and 
$H$ is a split Cartan subgroup. Fix $\lambda\in\h^*$.  Suppose
$\pi$ is a genuine principal series representation of $\tG$ with
infinitesimal character $\lambda$. By Example \ref{ex:verysimple} there
is such a representation, determined by its central character.  Let
$h=exp(\pi i\lambda)$ and
$$
G'_\C=Cent_{G_\C}(h^2).
$$
Let
$\theta'=int(h)$
considered as an involution of $G'_\C$.
Note that $G'_\C=G_\C$ if $\lambda$ is half-integral.
Let $G'$ be the
corresponding real form, and let $\tG'$ be an admissible cover.
Let $K'_\C$ be the fixed points of $\theta'$ acting on $G'_\C$.
Thus
$$
\Delta(\g',\h)=\Delta(2\lambda),\quad
\Delta({\mathfrak k}',\h)=\Delta(\lambda).
$$
The Cartan subgroup of $G'$ corresponding to
$\h$ is compact, and $G'$ has a discrete series representation with
infinitesimal character $\lambda$, determined by its central
character. Then $\pi$ and $\pi'$ have dual bigradings, and the map
$\pi\rightarrow \pi'$ extends to a duality of blocks.

For example if $\lambda$ is integral the
principal series representation is irreducible. Dually $G'$ is
compact, the trivial cover of $G'$ is admissible, and $\pi'$ is a
finite dimensional representation of $\tG'$.

If $\pi$ is any genuine irreducible representation of $\tG$ then we
may apply a sequence of Cayley transforms to obtain a minimal
principal series representation. This reduces the computation of the
dual of the block containing $\pi$ 
to the previous case; in particular $G'_\C$ and $G'$
are computed from the infinitesimal character as above.

\subsec{Example: Minimal Principal Series of Split Groups (continued)}

The duality of Theorem \ref{t:main} is for a single block, as in
\cite{ic4}. If $G$ is a real form of a connected reductive
algebraic group this duality can be promoted, roughly speaking, to a
duality on all blocks simultaneously. See \cite{abv} for details.
It would be very interesting to do this also in the nonlinear case.
We
limit our discussion here to minimal principal series of (simple) split
groups.

For simplicity we assume $G$ is the split real form of a connected,
semisimple complex group, and let $\tG$ be its admissible cover.
Fix an infinitesimal character for $G$. 
Then the genuine minimal principal series
representations of $G$ with this infinitesimal character are
parametrized by the genuine characters of $Z(\tG)$.

There are a finite number $\pi_1,\dots, \pi_n$ of such
representations, generating distinct blocks $\b_1,\dots,\b_n$. 
It follows from Section \ref{s:kl} that there are natural bijections
between these blocks, and the corresponding representations of the
extended Hecke algebra are isomorphic. Here are the number of such
blocks for the simple groups \cite[Section 6, Table 1]{shimura}:
\begin{equation}
\begin{aligned}
A_{2n}, E_6, E_8&: 1\\
A_{2n+1},D_{2n+1}, E_7&: 2\\
D_{2n}&: 4
\end{aligned}
\end{equation}

In fact it can be shown that the group of {\em outer automorphisms} of
$\tG$ acts transitively on $\{\pi_1,\dots,\pi_n\}$.

\subsec{Example: Discrete Series}
\label{s:discrete}

This is dual to the previous example. Suppose $G$ is a real form of a
connected, semisimple group, which contains a compact Cartan
subgroup. Let $\tG$ be the admissible cover of $G$, and fix an
infinitesimal character $\lambda$ for which $\tG$ has a genuine
discrete series representation. 

The number of genuine discrete series representations of $\tG$ with
infinitesimal character $\lambda$ depends on the real form.
If $G$ is quasisplit it can be shown, for example by a case-by-case
analysis, that the genuine discrete series representations of
$\tG$ with infinitesimal character $\lambda$ are in bijection with
the genuine principal series representation of the split real form of
$G$ discussed in the previous section.
However if $G$ is not quasisplit there may be fewer genuine discrete
series of $\wt{G}$.

For example suppose $G_\C=SL(2n,\C)$. As in Example \ref{s:gln} if
$G=SU(p,q)$ then $\tG$ has $1$ genuine discrete series representation with
infinitesimal character $\lambda$, if $p\ne q$, and $2$ if $p=q$.

\subsec{Example: Genuine Discrete Series for $E_7$}

There are three noncompact real forms of
$E_7$: split, Hermitian ($\mathfrak k=\R\times D_5$) and
{\it quaternionic} ($\mathfrak k=A_1\times A_5$). All have a compact
Cartan subgroup.
We label these $E_7(s), E_7(h)$ and $E_7(q)$, respectively.
Then at the appropriate infinitesimal character (for example
$\rho-\frac12\lambda_i$ where
$\lambda_i$ is an appropriate fundamental weight), $E_7(s)$ and $E_7(h)$
have two genuine discrete series representations. 
On the other hand $E_7(q)$ has only $1$ genuine discrete series
representation.

It is worth noting that $Z(\tG)=\Z/4\Z$ in the
split and Hermitian cases, and $\Ztwo\times\Ztwo$ in the quaternionic
real form. Let $\pi$ be a genuine discrete series representation of
$\tG$. If $G$ is split or quaternionic $\pi$ and $\pi^*$ have
distinct central characters, and are therefore the two genuine discrete
series representations. If $G$ is quaternionic then $\pi\simeq\pi^*$.

\subsec{Example: Genuine discrete series for $D_n$}

Finally we specialize the discussion in Example \ref{s:discrete} to
type $D_n$.

Write $\lambda=(\lam_1, \dots, \lam_n)$ in the usual coordinates.
Then $\lambda$ is half-integral
if
\[
\lambda_i\in\frac12\Z \quad\text{for all i}
\]
or
\[
\lam_i \in \pm \frac14 + \Z\quad\text{for all i.}
\]
In the first case
$\Delta(\lambda)$ is of type $D_p\times D_q$, $G=Spin(2p,2q)$, and
$\wt K\simeq Spin(2p)\times Spin(2q)$.
Write $\lambda=(a_1,\dots, a_p,b_1\,\dots, b_q)$ with $a_i\in\Z$ and
$b_j\in \Z+\frac12$. Then $\tG$ has discrete series representations
with Harish-Chandra parameter (with the obvious notation) 
$(a_1,\dots,a_{p-1}, \epsilon a_p;b_1,\dots,b_{q-1},\epsilon b_q)$ 
with $\epsilon=\pm1$. If $p=q$ it also has two more,
with Harish-Chandra parameter
$(b_1,\dots, b_{p-1},\epsilon b_p;a_1,\dots,a_{p-1},\epsilon a_p)$.
These two or four genuine discrete series representations have
distinct central characters.

In the second case
$\Delta(\lambda)$ is of type $A_{n-1}$ and
$\wt K$ is a two-fold cover of $U(n)$.
In this case $G$
the real form of $Spin(2n,\C)$ corresponding to the 
real form 
$\mathfrak s\mathfrak o^*(2n)$ of 
$\mathfrak s\mathfrak o(2n,\C)$, and a two-fold cover of $SO^*(2n)$.

\subsec{Example: $\mathrm{G}_2$}
\label{e:g2}
Let $G$ be the connected split real group of type $G_2$, and let $\tG$
an admissible two-fold cover of $G$ (which is unique up to
isomorphism).  
Since all real roots for $\tG$ are metaplectic in the sense of
Definition \ref{d:meta},
most of the results of this paper apply to this case.  
We will use Proposition
\ref{p:geom} to visualize the duality of Theorem \ref{t:main}.  The
closure order for the orbits of $K_\C$ (or $(\tK)_\C$) on the flag
variety $\frB$ for $\frg$ are depicted in Figure \ref{f:1}.  The
darkened nodes in the figure are irrelevant at this point, and will be
explained in a moment.

If we fix trivial infinitesimal character $\rho$ and the block
for $G$ containing the trivial representation, the fibers of the map
$\supp$ (described before Proposition \ref{p:geom}) for $G$ are all
singletons, with the exception of the open obit which supports
three representations.  The duality of \cite{ic4}
interchanges these three with the three
discrete series (supported on the three closed orbits) and interchanges
the unique representations supported on each of the intermediate orbits
in a way consistent with the obvious symmetry of those orbits
in the closure order of Figure \ref{f:1}.

Now consider the unique block $\caB$ of genuine representations of
$\tG$ with infinitesimal character $\rho/2$.  The image of
the injective map $\supp$ in this case is represented by the darkened nodes
in Figure \ref{f:1}.  The duality of Theorem \ref{t:main} takes
$\caB$ to itself and 
corresponds to an order-reversing involution of the
darkened  nodes.

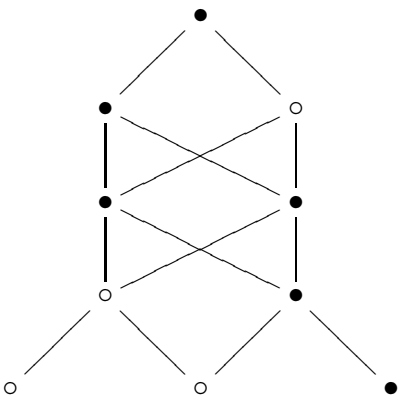
\begin{figure}
\label{f:1}
$$
{\xymatrixcolsep{2pc}
\xymatrixrowsep{2pc}
\xymatrix{
&&\bullet\\
&\bullet\ar@{-}[ur]
&&\circ\ar@{-}[ul]\\
&\bullet\ar@{-}[urr]\ar@{-}[u]&&\bullet\ar@{-}[ull]\ar@{-}[u]\\
&\circ\ar@{-}[urr]\ar@{-}[u]&&\bullet\ar@{-}[ull]\ar@{-}[u]\\
\circ\ar@{-}[ur]
&&\circ\ar@{-}[ur]\ar@{-}[ul]&&\bullet\ar@{-}[ul]
}
}
$$
\caption{The closure order for $K_\C$ orbits on the flag variety
of type $G_2$.}
\end{figure}

\def\cprime{$'$} \def\cftil#1{\ifmmode\setbox7\hbox{$\accent"5E#1$}\else
  \setbox7\hbox{\accent"5E#1}\penalty 10000\relax\fi\raise 1\ht7
  \hbox{\lower1.15ex\hbox to 1\wd7{\hss\accent"7E\hss}}\penalty 10000
  \hskip-1\wd7\penalty 10000\box7}
  \def\cftil#1{\ifmmode\setbox7\hbox{$\accent"5E#1$}\else
  \setbox7\hbox{\accent"5E#1}\penalty 10000\relax\fi\raise 1\ht7
  \hbox{\lower1.15ex\hbox to 1\wd7{\hss\accent"7E\hss}}\penalty 10000
  \hskip-1\wd7\penalty 10000\box7}
  \def\cftil#1{\ifmmode\setbox7\hbox{$\accent"5E#1$}\else
  \setbox7\hbox{\accent"5E#1}\penalty 10000\relax\fi\raise 1\ht7
  \hbox{\lower1.15ex\hbox to 1\wd7{\hss\accent"7E\hss}}\penalty 10000
  \hskip-1\wd7\penalty 10000\box7}
  \def\cftil#1{\ifmmode\setbox7\hbox{$\accent"5E#1$}\else
  \setbox7\hbox{\accent"5E#1}\penalty 10000\relax\fi\raise 1\ht7
  \hbox{\lower1.15ex\hbox to 1\wd7{\hss\accent"7E\hss}}\penalty 10000
  \hskip-1\wd7\penalty 10000\box7}
  \def\cftil#1{\ifmmode\setbox7\hbox{$\accent"5E#1$}\else
  \setbox7\hbox{\accent"5E#1}\penalty 10000\relax\fi\raise 1\ht7
  \hbox{\lower1.15ex\hbox to 1\wd7{\hss\accent"7E\hss}}\penalty 10000
  \hskip-1\wd7\penalty 10000\box7}

\end{document}